\documentclass[11pt]{amsart}
\usepackage{latexsym, amsmath, amssymb,amsthm,amsopn,amsfonts}
\usepackage{version}
\usepackage{epsfig,graphics,color,graphicx,graphpap}
\usepackage{amssymb}
\usepackage{todonotes}
\usepackage{listings}
\usepackage{mathrsfs}
\usepackage[citecolor=blue]{hyperref}
\usepackage[normalem]{ulem}
\usepackage{soul,xcolor}

\usepackage[framemethod=tikz]{mdframed}
\newcommand{\redsout}{\bgroup\markoverwith{\textcolor{red}{\rule[0.5ex]{2pt}{.4pt}}}\ULon}

\newcommand{\LV}{\left|}
\newcommand{\RV}{\right|}
\newcommand{\LC}{\left(}
\newcommand{\RC}{\right)}

\newcommand{\LA}{\left<}
\newcommand{\RA}{\right>}
\newcommand{\p}{\partial}

\numberwithin{equation}{section}
\newtheorem{theorem}{Theorem}[section]
\newtheorem{corollary}[theorem]{Corollary}
\newtheorem{proposition}[theorem]{Proposition}
\newtheorem{lemma}[theorem]{Lemma}
\newtheorem{definition}{Definition}[section]
\newtheorem{hypothesis}[definition]{Hypothesis}
\newtheorem{remark}{Remark}[section]
 
\newcommand{\R}{\mathbb R}

\usepackage{url}
\definecolor{mycolor}{rgb}{0.122, 0.435, 0.698}
\definecolor{aliceblue}{rgb}{0.94, 0.97, 1.0}

\newmdenv[innerlinewidth=0.5pt, roundcorner=4pt,linecolor=mycolor,innerleftmargin=6pt,
innerrightmargin=6pt,innertopmargin=6pt,innerbottommargin=6pt]{mybox}
\newmdenv[backgroundcolor=aliceblue,innerlinewidth=0.5pt, roundcorner=4pt,linecolor=mycolor,innerleftmargin=6pt,
innerrightmargin=6pt,innertopmargin=6pt,innerbottommargin=6pt]{mybox1}
\author[Lai]{Ru-Yu Lai}
\address{School of Mathematics, University of Minnesota, Minneapolis, MN 55455, USA}
\email{rylai@umn.edu}

\author[Zhou]{Hanming Zhou}
\address{Department of Mathematics, University of California Santa Barbara, Santa Barbara, CA 93106, USA}
\email{hzhou@math.ucsb.edu}
 
 
\title[Inverse problems for nonlinear transport equations]
{Inverse problems for time-dependent nonlinear transport equations}
\begin{document}
\setstcolor{red}

\begin{abstract}
In this work, we investigate inverse problems of recovering the time-dependent coefficient in the nonlinear transport equation in both cases: two-dimensional Riemannian manifolds and Euclidean space $\R^n$, $n\geq 2$. Specifically, it is shown that its initial boundary value problem is well-posed for small initial and incoming data. Moreover, the time-dependent coefficient appearing in the nonlinear term can be uniquely determined from boundary measurements as well as initial and final data. To achieve this, the central techniques we utilize include the linearization technique and the construction of special geometrical optics solutions for the linear transport equation. This allows us to reduce the inverse coefficient problem to the inversion of certain weighted light ray transforms.
Based on the developed methodology, the inverse source problem for the nonlinear transport equation in the scattering-free media is also studied.
\end{abstract}

\maketitle

\section{Introduction}
We study inverse problems for the transport equation with nonlinear absorption in both the Euclidean and Riemannian settings.  
Let $M$ be the interior of a compact non-trapping Riemannian manifold $(\overline M, g)$, of dimension $n\geq 2$, with smooth strictly convex (with respect to the metric $g$) boundary $\p M$.  
Let $TM$ be the tangent bundle of $M$. We denote the unit sphere bundle of the manifold $(M,g)$ by  
$$
SM :=\{(x,v)\in TM:\, |v|^2_{g(x)}:=\left< v,v\right>_{g(x)}=1\},
$$
where $\langle \cdot\,,\,\cdot \rangle_{g(x)}$ is the inner product on the tangent space $T_xM$.
Let $\p_+ S M$ and $\p_- SM$ be the outgoing and incoming boundaries of $SM$, respectively, and they are defined by
\begin{align*}
	\p_\pm SM := \{(x,v)\in S\overline{M}:\, x\in \p M,\, \pm \langle \nu(x), v\rangle_{g(x)} > 0\},
\end{align*}
where $\nu(x)$ is the unit outer normal vector at $x\in\p M$.
For any point $x\in M$, let $$S_xM:=\{v:\, (x,v)\in SM\}.$$ Moreover, we denote 
$$S M^2:=\{(x,v, v'): x\in M,\; v,\, v'\in S_xM \}.$$
Let $T>0$, we also denote  
$$
    M_T:=(0,T)\times M,\quad SM_T:=(0,T)\times SM\quad\hbox{and}\quad \p_\pm SM_T:=(0,T)\times \p_\pm SM.
$$

Let $X$ be the geodesic vector field which generates the geodesic flow on $SM$. In particular, $X=v\cdot\nabla$ is the directional derivative with respect to the $x$-variable in the Euclidean case. We consider the following initial boundary value problem for the nonlinear transport equation with a power-type nonlinear term:
\begin{align}\label{EQN: RTE equ}
	\left\{\begin{array}{rcll}
		\p_tf + X f  + \sigma f +q f^m &=& K(f) & \hbox{in } SM_T,  \\
		f &=&h_0  & \hbox{on } \{0\}\times SM,\\
		f &=&h_- & \hbox{on }\p_-SM_T,
	\end{array}\right.
\end{align}
where $m\geq 2$, $f\equiv f(t,x,v)$ is the solution, $\sigma\equiv \sigma(x,v)$ is the real-valued absorption coefficient, $q\equiv q(t,x)$ is the real-valued nonlinear coefficient and $K$ is the scattering operator, defined by
$$
    K(f)(t,x,v) := \int_{S_xM} \mu(x,v,v') f(t,x,v')\,dv'
$$
with the real-valued scattering coefficient $\mu$.

We define the set  
\begin{equation}\label{set M}
    \begin{split}
    \Omega:=\Big\{&(\sigma,\mu)\in L^\infty(SM)\times L^\infty(SM^2): 0\leq  \sigma(x,v) \leq \sigma_0, \quad 0 \leq \mu(x,v,v') \leq \mu_0,  \\  
    &\quad  
     \int_{S_xM} \mu(x,v,v') dv'\leq \sigma(x,v), \quad  \int_{S_xM} \mu(x,v',v) dv'\leq \sigma(x,v)\Big\}
     \end{split}
\end{equation}
for some positive constants $\sigma_0$ and $\mu_0$. These conditions are required in the proof of the well-posedness for the nonlinear transport equation in Section~\ref{sec:forward}.   
More explicitly, Theorem~\ref{THM:WELL} shows that there exists a constant $\delta>0$ such that given any \begin{equation}\label{DEF:set X}
	\begin{split}
	(h_0, h_-) \in \mathcal{X}_\delta:=  \{(h_0, h_-)\in L^\infty(SM)\times L^\infty(\p_-SM_T):\, \|h_0\|_{L^\infty(SM)}\leq \delta, \quad \|h_-\|_{L^\infty(\p_-SM_T)}\leq \delta\},
	\end{split}
\end{equation}
the problem \eqref{EQN: RTE equ} has a unique solution $f\in L^\infty(SM_T)$. Therefore, the measurement operator $\mathcal{A}_q:\mathcal{X}_\delta \rightarrow L^\infty(SM)\times L^\infty(\p_+SM_T)$, defined by
\begin{align*}
\mathcal{A}_q(h_0,h_-) =  (f|_{t=T},f|_{\p_+SM_T}),  
\end{align*}
is well-defined for such  $(h_0,h_-)\in \mathcal{X}_\delta$.

The main goal of this paper is to recover the nonlinear coefficient $q$, depending on both $t$ and $x$ variables, from the operator $\mathcal{A}_q$.  The choice of such measurement operator $\mathcal{A}_q$ is based on \cite{bellassouedInverseProblemLinear2019}, which shows that it is impossible
to recover the absorption over the whole domain with trivial initial data, that is, $f(0, x, v) \equiv 0$. In
other words, without the knowledge of the initial data, the coefficient in the cloaking region
can not be uniquely determined by solely using boundary measurements. 

Driven by its wide applications, there have been fruitful results for inverse problems for the transport type equations. We first review some relevant studies.
In the absence of the nonlinear term (i.e., $q\equiv 0$), the equation \eqref{EQN: RTE equ} is known as the linear Boltzmann equation or radiative transfer equation (RTE). Its corresponding inverse problem is sought to recover both absorption and scattering coefficients in the RTE from the albedo operator $\mathcal{A}_{\sigma,\mu}:f|_{\p_-SM_T}\mapsto f|_{\p_+SM_T}$. This setting is related to applications, such as medical imaging, remote sensing, and atmospheric science, and has been studied in different theoretical perspectives. They include unique determination of coefficients in \cite{IKun, CS1, CS2, CS3, CS98, SU2d}, stability estimates in \cite{Bal14, Bal10,  Bal18,BalMonard_time_harmonic, Machida14, Wang1999, ZhaoZ18} and the Riemannian setting in  \cite{AY15, MST10, MST10stability, MST11, McDowall04}. For more extensive discussions on the development of the related problems and methodologies, we refer the readers to  \cite{Balreview, Stefanov_2003} and the reference therein.
As for the nonlinear Boltzmann equations, with the knowledge of the albedo operator, the recovery of time-independent collision kernel were investigated in \cite{LaiUhlmannYang} for the stationary equation and in \cite{LiOuyang2022} for the dynamic one. Later, the setting of the time-dependent kernel was studied in \cite{LY2023}. A related work by using the source-to-solution map to recover the metric is addressed in \cite{BKLL21}.

The question of identifying time-dependent unknowns in the dynamic equations appears naturally from the practical applications and mathematical interests. Although most of the existing theoretical works mentioned above for the inverse transport problem concern the case when the coefficients depend only on spatial variables or velocity, it is important to understand and tackle the new challenges brought by the time $t$ during the reconstruction process. 
Motivated by the studies in \cite{bellassouedInverseProblemLinear2019,BELLASSOUED19} in which time-dependent absorption and scattering coefficients are stably recovered from the observations, we are interested in the questions of recovering the power type nonlinearity $q(t,x)$.

\subsection{Main result}
Given $(x,v)\in S\overline{M}$, let $\gamma_{x,v}(s)$ be the unique maximal geodesic satisfying the initial conditions 
$$
\gamma_{x,v}(0)=x,\quad \dot \gamma_{x,v}(0)=v,
$$
and it is defined on a maximal interval of existence $[-\tau_-(x,v),\tau_+(x,v)]$, see Section~\ref{Sec:notation} for more details. Here $\tau(x,v):=\tau_+(x,v)> 0$ is the forward exit time of the geodesic $\gamma_{x,v}$ so that $\gamma_{x,v}(\tau(x,v))\in\p M$. Since $\overline M$ is non-trapping, any maximal geodesic will exit $\overline M$ in finite time, i.e. $\tau(x,v)<\infty$ for any $(x,v)\in S\overline{M}$.

\begin{definition}
    We define the attenuated X-ray transform of a function $S\in L^\infty(M)$ with attenuation $\omega\in L^\infty(SM)$ by 
$$
    I_{\omega} S(x,v):=\int_0^{\tau(x,v)} S(\gamma_{x,v}(s)) e^{-\int_0^s \omega(\gamma_{x,v}(r),\dot{\gamma}_{x,v}(r))\,dr}\,ds,\quad
    \quad (x,v)\in \p_-SM.
$$
We say that $I_\omega$ is injective if $I_\omega S(x,v)=0$ for almost everywhere (a.e.) $(x,v)\in \p_-SM$ implies that $S\equiv 0$.
\end{definition}
The attenuated X-ray transform has been extensively studied, see e.g. survey papers \cite{Finch2003, Kuchment2006} for the Euclidean case and \cite{ZhouSurvey2024} for the Riemannian case, and the references therein.

We are now ready to present our main theorems. 
\begin{theorem}(Euclidean setting)\label{thm:main thm euclidean}
    Let $M$ be an open bounded and strictly convex domain in $\R^n$, $n\geq 2$, with smooth boundary. Let $(\sigma,\mu)\in \Omega$ so that the attenuated X-ray transform $I_{(m-1)\sigma}$ is injective. 
     Assume that $q_1$ and $q_2$ are in $L^\infty(M_T)$.
     Then $\mathcal A_{q_1} =\mathcal A_{q_2}$ on $\mathcal{X}_\delta$ implies 
     $$
         q_1=q_2\quad \hbox{ in } M_T.
     $$ 
\end{theorem}

For the Riemannian setting, we establish a similar uniqueness result in dimension 2, under an additional assumption on the scattering coefficient $\mu$.
 
\begin{theorem}(Riemannian setting)\label{thm:main thm geometric}
    Let $M$ be the interior of a compact non-trapping Riemannian manifold with smooth strictly convex boundary $\p M$, of dimension $\text{dim}\,M=2$. Let $(\sigma,\mu)\in\mathcal{M}$, defined by \eqref{set:mu} in subsection~\ref{sec:GO M}, so that the attenuated X-ray transform $I_{(m-1)\sigma}$ is injective.  
     Assume that $q_1$ and $q_2$ are in $L^\infty(M_T)$.
     Then $\mathcal A_{q_1} =\mathcal A_{q_2}$ on $\mathcal{X}_\delta$ implies 
     $$
         q_1=q_2\quad \hbox{ in } M_T.
     $$ 
\end{theorem}
Roughly speaking, the set $\mathcal M\subset \Omega$ imposes certain symmetry and vanishing conditions on $\mu$,
see subsection \ref{sec:GO M} for details. In particular, see Remark \ref{examples} for examples.

An analogous setting, motivated by the two-photon photoacoustic tomography, is investigated by the authors and Uhlmann \cite{LaiUhlmannZhou22} for both inverse coefficient and source problems for the nonlinear dynamic transport equation with time-independent coefficients. The main methodology is based on deriving Carleman estimates, a weighted $L^2$ estimate, for the linear transport equation in both Euclidean and Riemannian settings \cite{BK1981, carleman, Gaitan14,  Yamamoto2016, KlibanovP2006, Klibanov08, Lai-L, Machida14}. However, by this approach, only time-independent coefficients can be identified from the measurements. Therefore, in this paper, we seek for an alternative method that is applicable to take care of the time variable in the nonlinear term.

There are two main ingredients in the proof of Theorem~\ref{thm:main thm euclidean} and Theorem~\ref{thm:main thm geometric}.
The first ingredient is the construction of geometric optics (GO) solutions for the linear transport equations, which was first developed in \cite{bellassouedInverseProblemLinear2019,BELLASSOUED19} for the determination of time-dependent absorption coefficient and scattering coefficient for the linear Boltzmann equation. The remainder term of such GO solution decays to zero in $L^2$ norm with respect to some frequency parameter. In the current work, the construction is slightly different from the approach in \cite{bellassouedInverseProblemLinear2019,BELLASSOUED19,Rezig}. A key feature of our GO solutions is that they are defined on the infinite cylinder $\mathbb R\times M$, instead of the infinite slab $(0,T)\times \mathbb R^n$. Consequently, we do not need to embed $M$ into a larger domain or manifold. See subsection~\ref{sec:GO M} for additional remarks.

Theorem \ref{thm:main thm geometric} can be viewed as a generalization of Theorem \ref{thm:main thm euclidean} to nontrivial geometry. The proof of the decay of the remainder term of GO solutions in the Euclidean case relies essentially on the existence of global coordinates on $SM=M\times \mathbb S^{n-1}$. Unlike the Euclidean case, there are no global coordinates on Riemannnian manifolds in general. However, on two dimensional Riemannian manifolds, we can apply the global isothermal coordinates to characterize the set $\mathcal M$ for the scattering coefficient $\mu$, in order to derive similar decay property for the remainder term of GO solutions.

The second ingredient is the injectivity of certain weighted light ray transforms. With the help of the GO solutions, the determination of the nonlinear coefficients can be reduced to the study of the invertibility of the following \textit{attenuated light ray transform} $L_\omega S$ of a function $S\in L^\infty(\R\times M)$  
with attenuated coefficient $\omega \in L^\infty(SM)$: 
$$
    L_{\omega}S (t,x,v):=\int_0^{\tau(x,v)}S(t+s,\gamma_{x,v}(s))e^{-\int_0^s \omega(\gamma_{x,v}(r),\dot{\gamma}_{x,v}(r))\,dr}\,ds,\quad
    \quad (t,x,v)\in \R\times \p_- SM.
$$
When $\omega\equiv 0$, this is the standard \textit{light ray transform}, denoted by $LS$, whose injectivity was first established in \cite{Stefanov89} for Minkowski spacetime. The injectivity of $L$ was later generalized to the cases of static \cite{FIKO2021, FK2019} and stationary \cite{FIO2021} Lorentzian manifolds.
By taking the Fourier transform of the light ray transform with respect to the time variable $t$, we further reduce the problem to the stationary case, namely, the attenuated X-ray transform.
To the best of our knowledge, this is the first injectivity result for the light ray transform with non-trivial weight in the non-analytic category. Weighted light ray transforms on real-analytic Lorentzian manifolds were studied in \cite{Stefanov2017}. On the other hand, most results for the light ray transform in the literature are concerned with functions or tensors with compact support, we also address the non-compact case in the Appendix \ref{appendix 2}.

Under suitable geometrical conditions on the manifolds in different dimensions, the assumption of the injectivity of $I_{(m-1)\sigma}$ in Theorem~\ref{thm:main thm euclidean} and Theorem~\ref{thm:main thm geometric} can be removed. Therefore, we have the following corollary that follows immediately from Theorem~\ref{thm:main thm euclidean} and Theorem~\ref{thm:main thm geometric}. A compact Riemannian manifold with smooth strictly convex boundary is called {\it simple} if it is simply connected and free of conjugate points \cite{PSUbook}.
\begin{corollary}\label{cor:geometric condition}
    Let $(\sigma,\mu)\in \Omega$. Assume that $q_1$ and $q_2$ are in $L^\infty(M_T)$.
    Suppose that $\mathcal A_{q_1} =\mathcal A_{q_2}$ on $\mathcal{X}_\delta$.
    Then the uniqueness results hold for the following cases:
    \begin{enumerate}
        \item For $n\geq 3$, assume that $M$ is an open bounded and strictly convex domain in $\R^n$, and $\sigma\in C^\infty(S\overline M)$, then $q_1=q_2$ in $M_T$; 
        \item For $n=2$, assume that either
        \begin{enumerate}
            \item $M$ is an open bounded and strictly convex domain in $\R^2$; or
            \item  $M$ is the interior of a simple surface and $(\sigma,\mu)\in \mathcal{M}$. 
        \end{enumerate}  
        Then if $\sigma\in C^\infty(\overline M)$, we have $q_1=q_2$ in $M_T$.
    \end{enumerate}
    
\end{corollary}

Finally, when the system is scattering free, i.e. $\mu=0$, Theorem \ref{thm:main thm geometric} holds in any dimension.
\begin{theorem}\label{thm:riemannian with scattering}
    Let $M$ be the interior of a compact non-trapping Riemannian manifold with smooth strictly convex boundary $\p M$, of dimension $dim\,M\geq 2$. Let $(\sigma,0)\in \Omega$ so that the attenuated X-ray transform $I_{(m-1)\sigma}$ is injective.
     Assume that $q_1$ and $q_2$ are in $L^\infty(M_T)$, then $\mathcal A_{q_1} =\mathcal A_{q_2}$ on $\mathcal{X}_\delta$ implies 
     $$
         q_1=q_2\quad \hbox{ in } M_T.
     $$ 
\end{theorem}

This paper is organized as follows. In Section~\ref{sec:forward}, we introduce necessary notations and establish well-posedness theorem for the nonlinear transport equations. We also construct geometric optics (GO) solutions for the linear transport equation on Euclidean domains and two dimensional Riemannian manifolds. Section~\ref{sec:ICP} is devoted to the proof of Theorem \ref{thm:main thm euclidean} and \ref{thm:main thm geometric} by the linearization scheme and the invertibility of weighted light ray transforms. We also show determination results under the monotonicity condition $q_1\leq q_2$. As an immediate application, we study the inverse source problem in the absence of scattering in Section~\ref{sec:inverse source problem}, which is equivalent to the attenuated light ray transform. In Appendix \ref{appendix 2}, we address the invertibility of the light ray transform of functions or tensor fields that are not necessarily supported in $M_T$.

\section{Preliminaries and forward problem}\label{sec:forward}
In this section, we introduce necessary notations first and then study the well-posedness problem for the initial boundary value problem for the nonlinear transport equation \eqref{EQN: RTE equ}. Moreover, we also construct the special solutions, called geometric optics (GO) solutions, for the linear transport equation. This will be one of key ingredients in solving the inverse coefficient problems later.

\subsection{Notations}\label{Sec:notation}
Recall that we denote the forward exit time from $(x,v)\in S\overline M$ by $\tau(x,v)$, which is defined as follows:
$$\tau(x,v):=\sup\{\tilde s : \gamma_{x,v}(s)\in M \, \mbox{ for }\,  0\leq s<\tilde{s}\}<\infty.$$ 
Recall the assumption that $M$ is non-trapping. Similarly, we also define the backward exit time by
$$\tau_-(x,v):=\sup\{\tilde s : \gamma_{x,v}(-s)\in M \, \mbox{ for }\, 0\leq s<\tilde{s}\}<\infty.$$
Thus, $\gamma_{x,v}(\tau(x,v))\in \p M$ and $\gamma_{x,v}(-\tau_-(x,v))\in \p M$.
In particular, they satisfy  
$\tau(x,v)=\tau_-(x,-v)$ for all $(x,v)\in S \overline{M} $ and 
$\tau_-(x,v)|_{\p_-SM}=\tau(x,v)|_{\p_+SM}=0$.

Denote the geodesic flow through $(x,v)\in S\overline{M}$ by 
$$\rho_{x,v}(t) :=(\gamma_{x,v}(t),\dot\gamma_{x,v}(t)),\quad \rho_{x,v}(0)=(x,v).$$
Let $X$ be the generating vector field of the geodesic flow $\rho_{x,v}(t)$, that is, for a given function $f$ on $SM$, $Xf(x,v)=\frac{d}{dt}f(\rho_{x,v}(t) )|_{t=0}$. Notice that in the Euclidean space $\R^n$, $\rho_{x,v}(t) =(x+tv,\,v)$ and $X=v\cdot\nabla_x$, where $v$ is independent of $x$.

We define the spaces $L^p(SM)$ and $L^p(SM_T)$, $1\leq p<\infty$, with the norm
$$
    \|f\|_{L^p(SM)}= \LC\int_{SM} |f|^p\,d{\Sigma}\RC^{1/p} \quad\hbox{and}\quad \|f\|_{L^p(SM_T)}= \LC\int^T_0\int_{SM} |f|^p\,d{\Sigma}dt\RC^{1/p},
$$ 
with $d\Sigma=d\Sigma(x,v)$ the volume form of $SM$.
Moreover, for the spaces $L^p(\p_\pm SM_T)$, we define their norm to be
$$
\|f\|_{L^p(\p_\pm SM_T)}:=\|f\|_{L^p(\p_\pm SM_T; d\xi)}= \LC\int^T_0\int_{\p_\pm SM} |f|^p\, d\xi dt\RC^{1/p},
$$ 
where $d\xi(x,v):=|\langle \nu(x),v\rangle_{g(x)}|d\tilde{\xi}(x,v)$ with $d\tilde{\xi}$ the standard volume form of $\p SM$. 
Hereafter, we will drop the subindex $g(x)$ in $\langle \nu(x),v\rangle_{g(x)}$ if no confusion arises. When $p=\infty$, $L^\infty(SM)$, $L^\infty(SM_T)$ and $L^\infty(\p_\pm SM_T)$ are the standard vector spaces consisting of all functions that are essentially bounded.

\subsection{The forward problems}\label{sec:wellposedness}
The following well-posedness results for the linear transport equation was proven in \cite{LaiUhlmannZhou22}.
\begin{proposition}[Well-posedness for linear transport equation]\label{prop:forward problem linear}
	Suppose that $(\sigma,\mu)\in\Omega$. Let $S\in L^\infty(SM_T)$, $h_0\in L^\infty(SM)$ and $h_-\in L^\infty(\p_-SM_T)$.
	We consider the following problem:
	\begin{align}\label{EQN: linear RTE source}
	\left\{\begin{array}{rcll}
		\p_tf + X f + \sigma f &=& K(f) + S & \hbox{in }SM_T,  \\
		f &=&h_0 & \hbox{on }\{0\}\times SM,\\
		f &=& h_- & \hbox{on }\p_-SM_T.
	\end{array}\right.
\end{align}
Then \eqref{EQN: linear RTE source} has a unique solution $f$ in $L^\infty(SM_T)$ satisfying
\begin{align}\label{stability}
	\|f\|_{L^\infty(SM_T)}\leq C\LC\|h_0\|_{L^\infty(SM)} + \|h_-\|_{L^\infty(\p_-SM_T)} + \|S\|_{L^\infty(SM_T)}\RC,
\end{align}
where the constant $C$ depends on $\sigma$ and $T$.
\end{proposition}
\begin{remark} 
    Note that the coefficients $\sigma$ and $\mu$ are assumed to be real-valued functions. Hence, for complex-valued $h_0,\,h_-$ and $S$, by splitting the real and imaginary parts, the unique existence of the solution $f$ to \eqref{EQN: linear RTE source} still holds.
\end{remark}

In addition, as shown below, the solution $f$ also satisfies the estimate in $L^p$ norm. This property will be applied to prove the existence of the GO solutions later where we only apply the case $p=2$.  
\begin{corollary}
    Let $p\in [1,\infty)$. Under the same hypothesis of Proposition~\ref{prop:forward problem linear}, there exists a constant $C>0$, independent of $S$, $h_0$ and $h_-$, such that the solution $f$ satisfies
    \begin{align}\label{stability Lp}
	\|f \|_{L^p(SM_T)}+\|f\|_{L^p(\p_+SM_T;d\xi)}\leq C\LC\|h_0\|_{L^p(SM)} + \|h_-\|_{L^p(\p_-SM_T;d\xi)} + \|S\|_{L^p(SM_T)}\RC.
\end{align}
\end{corollary}
\begin{proof}
    Let $f=f_{R}+if_{I}$, where $f_R$ and $f_I$ denote the real and imaginary parts of $f$. From Proposition~\ref{prop:forward problem linear}, we get $f\in L^p(SM_T)$ and thus $|f|^{p-2}\overline{f}\in L^q(SM_T)$ for ${1\over p}+{1\over q}=1$. 
Observe that  
\begin{align*}
    (\p_t+X) (|f|^p) &= p|f|^{p-1}(\p_t+X)(|f|) \\
    &= p|f|^{p-2}(f_R(\p_t+X)(f_R)+f_I(\p_t+X)(f_I))\\
    &=  \text{Re}\LC p|f|^{p-2} \overline{f} (\p_t+X)(f) \RC,
\end{align*}
where $\text{Re}(u)$ represents the real part of a complex-valued function $u$. See \cite[Theorem~ 6.17]{LiebLoss} for the derivative of the absolute value of a function.
Hence, for any fixed $t\in (0,T)$, multiplying \eqref{EQN: linear RTE source} with $|f|^{p-2}\overline{f}$ and integrating over $SM$. Taking the real part leads to
\begin{align}\label{EST:f Lp}
         \int_{SM} \LC {1\over p} (\p_t +X)|f|^p(t) +  \sigma |f|^p(t)\RC \,d\Sigma = \text{Re}\LC\int_{SM} K(f) |f|^{p-2}\overline{f}(t)\,d\Sigma + \int_{SM} S|f|^{p-2}\overline{f}(t)\,d\Sigma\RC.
\end{align}
Here we denote $f(t):=f(t,\cdot,\cdot)$ to simply the expression. 

We first deal with the first term on the right-hand side of \eqref{EST:f Lp}. 
Following the arguments in \cite[Chapter XXI, Lemma~1]{Lion}, one can show that $K$ is a linear and continuous operator mapping from $L^p(SM)$ into itself for $p\in [1,\infty]$ and satisfies $\|K(f)(t)\|_{L^p(SM)}\leq \sigma_0 \|f(t)\|_{L^p(SM)}$ for any $t\in (0,T)$. 
Applying H\"older's inequality, \eqref{set M}, and $p=(p-1)q$ to derive the following estimate:
\begin{equation}\label{EST:Kf} 
\begin{split}
    \LV \int_{SM} K(f) |f|^{p-2}\overline{f}(t)\,d\Sigma\RV 
    &\leq \|K(f)(t)\|_{L^p(SM)} \LC\int_{SM} |f|^{q(p-1)}(t) \,d\Sigma\RC^{1\over q} \\
    &\leq \sigma_0\|f(t)\|_{L^p(SM)} \LC\int_{SM} |f|^{p}(t) \,d\Sigma\RC^{1\over q} \\ 
    &\leq \sigma_0\|f(t)\|_{L^p(SM)}\|f(t)\|_{L^p(SM)}^{p-1} \\
     &\leq \sigma_0\|f(t)\|^p_{L^p(SM)}.
     \end{split}
\end{equation}
For the second term on the right-hand side of \eqref{EST:f Lp}, Young's inequality yields that
\begin{align}\label{EST:S Lp}
     \int_{SM} S|f|^{p-2}\overline{f}(t)\,d\Sigma
     \leq {1\over p}\int_{SM} |S|^p(t)\,d\Sigma + {1\over q}\int_{SM} |f|^p(t)\,d\Sigma.
\end{align}
Moreover, we deduce from Green's formula \cite[Proposition~3.5.12 and Lemma~3.6.7]{PSUbook} that
\begin{align}\label{EST:Green}
    \int_{SM} X  |f|^p(t) \,d\Sigma
    = \int_{\p SM} |f|^p(t) \LA\nu(x),v\RA\,d\tilde{\xi}
    = \int_{\p_+SM} |f|^p(t) \,d\xi-\int_{\p_-SM} |f|^p(t) \,d\xi.
\end{align}
    Denote 
    $$
        E(t) = \int_{SM} |f|^p(t) \,d\Sigma, \quad 0\leq t\leq T.
    $$
    Combining \eqref{EST:f Lp} - \eqref{EST:Green} together yields
    \begin{equation}\label{EST:f derivative}
    \begin{split}
        &\quad E'(t)+ \int_{\p_+SM} |f|^p(t) \,d\xi + p\int_{SM} \sigma|f|^p(t)\,d\Sigma \\
        &\leq \int_{\p_-SM} |f|^p(t)\,d\xi + \int_{SM}|S|^p(t)\,d\Sigma + \LC {p\over q}+p\sigma_0\RC E(t).
        \end{split}
    \end{equation}
    Moreover, we integrate over the time interval $(0,t)$ and obtain
    \begin{align}\label{EST:energy}
         E(t)
         \leq C\int^t_0 E(s)\,ds + \|f \|^p_{L^p(\p_-SM_T)} + \|S \|_{L^p(SM_T)}^p + \|f(0)\|^p_{L^p(SM)},
    \end{align}
    where $C$ depends on $p,\,q$ and $\sigma_0$.
    The Gronwall's inequality implies that
    \begin{align*} 
        E(t) \leq C\LC\|f \|^p_{L^p(\p_-SM_T)} + \|S \|_{L^p(SM_T)}^p + \|f(0)\|^p_{L^p(SM)}\RC,
    \end{align*}
    for $0\leq t\leq T$, where the constant $C>0$ depends on $p,\,q,\,\sigma_0$ and $t$. This then gives 
    \begin{align}\label{EST: f(t)}
        \|f\|_{L^p(SM_T)}^p\leq C\LC\|f \|^p_{L^p(\p_-SM_T)} + \|S \|_{L^p(SM_T)}^p + \|f(0)\|^p_{L^p(SM)}\RC,
    \end{align}
    by integrating with respect to the time variable again over $(0,T)$.

    Finally, back to \eqref{EST:f derivative}, we integrate both sides over $(0,T)$ and
    utilize \eqref{EST: f(t)} to derive
    \begin{align*} 
         \|f(T)\|^p_{L^p(SM)} + \|f \|^p_{L^p(\p_+SM_T)} 
         \leq C\LC \|f \|^p_{L^p(\p_-SM_T)} + \|S \|_{L^p(SM_T)}^p + \|f(0)\|^p_{L^p(SM)}\RC .
    \end{align*}
    This further yields 
    \begin{align}\label{EST:f boundary} 
         \|f \|^p_{L^p(\p_+SM_T)} 
         \leq C\LC\|f \|^p_{L^p(\p_-SM_T)} + \|S \|_{L^p(SM_T)}^p + \|f(0)\|^p_{L^p(SM)} \RC.
    \end{align}
    The proof is complete by combining both estimates \eqref{EST: f(t)} and \eqref{EST:f boundary}.
     \end{proof}

The forward problem for the nonlinear transport equation can be established for small data. 
\begin{theorem}[Well-posedness for nonlinear transport equation]\label{THM:WELL}
	Let $q\in L^\infty(M_T)$. Suppose that $(\sigma,\mu)\in\Omega$.
	Then there exists a small parameter $0<\delta<1$ such that for any
	\begin{equation}\label{DEF:set X}
	\begin{split}
	(h_0,h_-) \in \mathcal{X}_\delta:=  \{(h_0,h_-)\in L^\infty(SM)\times L^\infty(\p_-SM_T):\, \|h_0\|_{L^\infty(SM)}\leq \delta, \quad \|h_-\|_{L^\infty(\p_-SM_T)}\leq \delta\},
	\end{split}
	\end{equation}
	the problem \eqref{EQN: RTE equ} has a unique solution $f\in L^\infty(SM_T)$ satisfying
	$$
	\|f\|_{L^\infty(SM_T)}\leq C \LC \|h_0\|_{L^\infty(SM)}+\|h_-\|_{L^\infty(\p_-SM_T)} \RC, 
	$$
	where the positive constant $C$ is independent of $f$, $h_0$ and $h_-$.
\end{theorem}
\begin{proof}
The unique existence of solution for \eqref{EQN: RTE equ} follows immediately by adapting the proof of Theorem~2.6 in \cite{LaiUhlmannZhou22} with the nonlinear term $q(t,x) u^m$.
\end{proof}

Based on the above well-posedness result, for any $(h_0,h_-)\in \mathcal{X}_\delta$ with sufficiently small $\delta$, the problem \eqref{EQN: RTE equ} has a small unique solution $f\in L^\infty(SM_T)$. Hence, the measurement map $\mathcal{A}_q:\mathcal{X}_\delta\rightarrow L^\infty(SM)\times L^\infty(\p_+SM_T)$ is well-defined.

\subsection{The adjoint problem}
In order to achieve the derivation of the integral identity in the next section, we will study the existence of solution for the adjoint problem:
\begin{align}\label{EQN: adjoint problem setup}
	\left\{\begin{array}{rcll}
		\p_t \tilde{f} + X \tilde{f} - \tilde\sigma \tilde{f} &=& -K^*(\tilde{f}) - \tilde{S}  & \hbox{in } SM_T,  \\
	    \tilde{f} &=&\tilde h_0& \hbox{on } \{T\}\times SM,\\
	    \tilde{f} &=&\tilde h_-& \hbox{on }\p_+SM_T,
	\end{array}\right.
\end{align}
where $K^*(\tilde f) (t,x,v):= \int_{S_xM} \tilde \mu(x,v',v) \tilde f(t,x,v')\,dv'$.

Suppose that $(\tilde\sigma,\tilde\mu)\in \Omega$ and $\tilde S\in L^\infty(SM_T)$. Here $\tilde \sigma$ and $\tilde\mu$ are real-valued functions and $\tilde S$ can be a complex-valued function. Let $\tilde h_0\in L^\infty(SM)$ and $\tilde h_-\in L^\infty(\p_+SM_T)$.
We denote
$$
     \sigma(x,v):=\tilde\sigma(x,-v),\quad  \mu(x,v',v):=\tilde \mu(x,-v,-v'),\quad  S(t,x,v) := \overline{\tilde S(T-t,x,-v)}. 
$$
Thus $(\sigma,\mu)\in \Omega$ and $S\in L^\infty(SM_T)$.
By Proposition~\ref{prop:forward problem linear}, there exists a unique solution $f\in L^\infty(SM_T)$ to the problem
\begin{align}\label{EQN: linear problem}
	\left\{\begin{array}{rcll}
		\p_tf  + X f  + \sigma f  &=& K(f) + S & \hbox{in } SM_T,  \\
		f  &=& h_0 & \hbox{on } \{0\}\times SM,\\
		f  &=& h_-& \hbox{on }\p_-SM_T,
	\end{array}\right.
\end{align}
where $h_0(x,v):=\overline{\tilde h_0(x,-v)}\in L^\infty(SM)$ and $h_-(t,x,v):=\overline{\tilde h_-(T-t,x,-v)}\in L^\infty(\p_-SM_T)$.

Moreover, let $\tilde f(t,x,v)=\overline{f(T-t,x,-v)}$, we observe that
\begin{align*}
    (\p_t + X )\tilde{f} (t,x,v) 
    &=  - \LC  \overline{(\p_t f )  (T-t,x,-v)}+  \overline{ (X f )(T-t,x,-v)} \RC\\
    &= \sigma(x,-v)\overline{f (T-t,x,-v)} -  \int \mu(x,-v, -v')\overline{f (T-t,x,-v')}\,dv'  - \overline{S(T-t,x,-v)}\\
    &= \tilde\sigma(x,v) \tilde{f}(t,x,v) - \int  \tilde\mu(x, v', v)\tilde{f}(t,x,v')\,dv' - \tilde{S}(t,x,v).
\end{align*}
This implies $\tilde f\in L^\infty(SM_T)$ is the solution to
$$
      \p_t \tilde{f} + X \tilde{f} - \tilde\sigma \tilde{f}= -K^*(\tilde{f}) - \tilde{S},
$$
and, moreover, it satisfies $\tilde{f}(T,x,v) = \overline{f(0,x,-v)}= \overline{h_0(x,-v)}=\tilde h_0(x,v)$ and $\tilde{f}|_{\p_+SM_T}=\tilde h_-$.
Also it satisfies the estimate  
\begin{align}\label{stability}
	\|\tilde f\|_{L^\infty(SM_T)}\leq C\LC\|\tilde h_0\|_{L^\infty(SM)} + \|\tilde h_-\|_{L^\infty(\p_+SM_T)} + \|\tilde S\|_{L^\infty(SM_T)}\RC.
\end{align}

\subsection{Construction of GO solutions}
The main goal here is to find a category of special solutions to the linear transport equations in both Euclidean and Riemannian settings. We will start by constructing such solutions for the Euclidean case in $\R\times M\times\mathbb{S}^{n-1}$. On the other hand, due to the dependence of variables $x$ and $v$ in the Riemannian case, the choice of the leading terms of the GO solutions will be different from the former case, see also Remark~\ref{remark:comparison} for more discussions.

\subsubsection{Construction of GO solutions in the Euclidean space}\label{sec:GO}
In the Euclidean space, let $M$ be an open and bounded strictly convex domain in $\R^n$ and let $g=e$ be the Euclidean metric. Then 
$$
    S_xM = \mathbb{S}^{n-1} \quad  \hbox{ and } \quad SM= M\times \mathbb{S}^{n-1}.
$$
The geodesic vector field acting on functions $f=f(x,v)$ on $SM$ will be 
$$
        Xf(x,v) = \Big.{d\over dt}f(x+tv,v)\Big|_{t=0} = v\cdot\nabla f(x,v).
$$

We consider the following linear transport equation in the Euclidean space:
$$
\p_tf + v\cdot\nabla f + \sigma f = K(f).
$$
First, we find special solutions to the equation $\p_t f+v\cdot\nabla f +\sigma f= 0$. To this end, given a real-valued function $\phi(t,x,v)\in C^\infty_0(\R, C^\infty_0(\p_- SM))$, let 
\begin{align}\label{def:varphi}
   \varphi(t,x,v) :=\phi(t-\tau_-(x,v),x-\tau_-(x,v)v,v).
\end{align}
Since $\tau_-(x+tv,v)=\tau_-(x,v)+t$, we have
\begin{align}\label{ID:time derivative}
    v\cdot \nabla \tau_-(x,v)=\Big.{d\over dt}\tau_-(x+tv,v)\Big|_{t=0} = \lim_{t\rightarrow 0}{\tau_-(x+tv,v) - \tau_-(x,v)\over t}= 1,
\end{align}
In addition, the flow $\rho_{x,v}(-\tau_-(x,v))=(x-\tau_-(x,v)v,v)$ is invariant along the line so that $v\cdot \nabla \rho_{x,v}(-\tau_-(x,v))=0$. Thus, $\varphi$ satisfies the following boundary value problem:
\begin{align*} 
	\left\{\begin{array}{rcll}
		\p_t\varphi+ v\cdot \nabla \varphi  &=& 0& \hbox{in } \R\times SM,  \\
		\varphi|_{\p_- SM}&=& \phi & \hbox{on } \R\times \p_-SM. 
    \end{array}\right.
\end{align*}
For $\lambda\neq 0$, we define the complex-valued functions $\varphi_\lambda$ by
$$
    \varphi_\lambda (t,x,v) := \varphi(t,x,v)e^{ i\lambda(t-x\cdot v)} , 
$$ 
where the subindex $\lambda$ is used to emphasize the dependence on $\lambda$. From direct computations, it is clear that $\varphi_\lambda$ are also solutions to $\p_t f+v\cdot \nabla f=0$. 

We next define the function  
$$
    \Theta_\sigma(x,v) := e^{-\int^{\tau_-(x,v)}_0\sigma(x-(\tau_-(x,v)-s)v,v)ds}.
$$
Since $\sigma(x-(\tau_-(x,v)-s)v,v)$ is also invariant along the line such that 
$$
    (v\cdot\nabla) \sigma(x-(\tau_-(x,v)-s)v,v)=0.
$$
Combining with \eqref{ID:time derivative}, we get that $\Theta_\sigma$ satisfies $(\p_t + v\cdot \nabla + \sigma) \Theta_{\sigma}(x,v)=0$.
Therefore, we deduce that the function
$
\varphi_\lambda \Theta_\sigma (t,x,v) 
$
is a solution to $(\p_t + v\cdot \nabla + \sigma)f=0$ in $\R\times SM$.

Standing on these facts, we are now ready to prove the existence of the GO solutions.

\begin{proposition}[GO solutions in Euclidean setting]\label{prop:GO}
Let $M$ be a bounded strictly convex domain in $\R^n$, $n\geq 2$. Let $(\sigma,\mu)\in \Omega$. For any $\lambda\neq 0$ and $\phi\in C^\infty_0(\R,C^\infty_0(\p_- SM))$, the linear transport equation 
$$
    \p_t u+ v\cdot \nabla u +\sigma u= K(u)
    $$
has solutions of the form 
\begin{align}\label{GO positive}
	u_\lambda (t,x,v) &= \varphi_\lambda (t,x,v) \Theta_\sigma (x,v) + r_\lambda(t,x,v)
\end{align}
in the space $L^\infty(SM_T)$.
Moreover,  
the remainder term $r_\lambda$ satisfies the estimate 
\begin{align}\label{bound r lambda}
    \|r_\lambda\|_{L^\infty(SM_T)}\leq C \sigma_0\|\phi\|_{L^\infty(\R\times\p_-SM)}.
\end{align}

In particular, the remainder $r_\lambda$ satisfies the following decay property:
\begin{align}\label{limit r positive}
\lim\limits_{\lambda \rightarrow \pm \infty} \|r_\lambda\|_{L^2(SM_T) } =0.
\end{align}
Here the constant $C>0$ depends only on $\sigma$ and $T$. 
\end{proposition}

\begin{proof}
We start by showing the existence of $r_\lambda$. By the definitions of $\varphi_\lambda$ and $\Theta_\sigma$, we deduce from $(\sigma,\mu)\in\Omega$ that
\begin{align*}
    |K(\varphi_\lambda \Theta_\sigma)(t,x,v)|
    &\leq \int_{\mathbb{S}^{n-1}} \mu(x,v,v') |\varphi(t,x,v')| \,dv' \leq \sigma_0 \|\phi\|_{L^\infty(\R\times\p_-SM)},\quad\hbox{ for }(t,x,v)\in SM_T,
\end{align*}
which implies
$K(\varphi_\lambda \Theta_\sigma)\in L^\infty(SM_T)$. With this, based on Proposition~\ref{prop:forward problem linear}, for a fixed $\lambda\neq 0$, there exists a unique solution $r_\lambda\in L^\infty(SM_T)$ to the problem
    \begin{align}\label{EQN: linear RTE}
	\left\{\begin{array}{rcll}
		(\p_t + X  + \sigma)r_\lambda &=& K(r_\lambda) + K(\varphi_\lambda \Theta_\sigma) & \hbox{in } SM_T,  \\
		r_\lambda &=&0  & \hbox{on } \{0\}\times SM,\\
		r_\lambda &=&0 & \hbox{on }\p_-SM_T,
	\end{array}\right.
\end{align}
satisfying
\begin{align*}
    \|r_\lambda\|_{L^\infty(SM_T)} 
    \leq C\|K(\varphi_\lambda \Theta_\sigma)\|_{L^\infty(SM_T)}\leq C \sigma_0 \|\phi\|_{L^\infty(\R\times\p_-SM)}.
\end{align*}

To show \eqref{limit r positive}, we first note that by \eqref{stability Lp} with $p=2$, the remainder $r_\lambda$ also satisfies 
\begin{align}\label{EST:r lambda}
    \|r_\lambda\|_{L^2(SM_T)} 
    \leq C \|K(\varphi_\lambda \Theta_\sigma)\|_{L^2(SM_T)}.
\end{align}
Next we write $K(\varphi_\lambda \Theta_\sigma)(t,x,v)=e^{ i\lambda t}\mathcal{K}_{\Phi,\lambda}(t,x,v)$, where we denote
$$
    \mathcal{K}_{\Phi,\lambda}(t,x,v) : =\int_{\mathbb{S}^{n-1}} e^{ -i\lambda x\cdot v'} \Phi(t,x,v,v') \,dv',
$$
and
$$
    \Phi(t,x,v,v') := \mu(x,v,v') \varphi(t,x,v')  e^{-\int^{\tau_-(x,v')}_0\sigma(x-(\tau_-(x,v')-s)v',v')ds} . 
$$
For any $t\in (0,T)$, $v\in \mathbb{S}^{n-1}$ and $x\in M \setminus \{0\}$, one can see that $e^{ -i\lambda x\cdot v'}$ is weakly convergent in $L^2(\mathbb{S}^{n-1})$ (see also \cite{bellassouedInverseProblemLinear2019}). This implies that the function $\mathcal{K}_{\Phi,\lambda}(t,x,v)$ converges to zero as $\lambda \rightarrow \pm\infty$. Hence, by dominated convergence theorem, we deduce that
$$
\lim_{\lambda\rightarrow \pm\infty}\|\mathcal{K}_{\Phi,\lambda}\|_{L^2(SM_T)}=0.
$$
Finally, combining with the facts that $\|K(\varphi_\lambda \Theta_\sigma)\|_{L^2(SM_T)} = \|\mathcal{K}_{\Phi,\lambda}\|_{L^2(SM_T)}$ and \eqref{EST:r lambda}, we complete the proof. 
\end{proof}

We also have GO solutions for the adjoint problem in the following theorem, which can be shown similarly. 
\begin{proposition}\label{prop:conjugate GO}
Suppose all hypothesis in Proposition~\ref{prop:GO} hold. The linear transport equation 
$$
    \p_t u+ v\cdot \nabla u  -\sigma u= -K^*(u)
$$
has solutions of the form 
\begin{align}\label{GO negative}
	u_\lambda (t,x,v) &= \varphi_\lambda (t,x,v) \Theta_{-\sigma} (t,x,v) + r_\lambda(t,x,v)
\end{align}
in the space $L^\infty(SM_T)$.
Here $K^*(f) (t,x,v):= \int_{S_xM} \mu(x,v',v) f(t,x,v')\,dv'$. 
Moreover,  
the remainder term $r_\lambda$ satisfies the estimate
$$
    \|r_\lambda\|_{L^\infty(SM_T)}\leq C \sigma_0\|\phi\|_{L^\infty(\R\times\p_-SM)}
$$
and has the decay property
\begin{align}\label{limit r negative}
\lim\limits_{\lambda \rightarrow \pm\infty} \|r_\lambda\|_{L^2(SM_T) } =0.
\end{align}
Here the constant $C>0$ depends on $\sigma$ and $T$. 
\end{proposition}

\subsubsection{Construction of GO solutions in the Riemannian manifold $(M,g)$}\label{sec:GO M}
In a similar spirit to the Euclidean setting above, we will construct GO solutions to the linear transport equation on a manifold $M$:
$$
\p_tf + X f + \sigma f = K(f).
$$
Given a real-valued function $\phi(t,x,v)\in C^\infty_0(\R,C^\infty_0(\p_- SM))$, let 
$$
   \varphi(t,x,v) :=\phi(t-\tau_-(x,v),\rho_{x,v}(-\tau_-(x,v))), 
$$
with the geodesic flow $\rho_{x,v}(t) =( \gamma_{x,v}(t),\dot{\gamma}_{x,v}(t))$. The definition of $\rho_{x,v}$ implies that $ \rho_{x,v}(-\tau_-(x,v))$ is invariant along the geodesic flow.
With $X\tau_-(x,v)=1$ on $SM$, we get that $\varphi$ is a solution to the following boundary value problem:
\begin{align*} 
	\left\{\begin{array}{rcll}
		\p_t\varphi+ X \varphi  &=& 0& \hbox{in } \R\times SM,  \\
		\varphi|_{\p_-SM}&=& \phi & \hbox{on } \R\times \p_-SM. 
    \end{array}\right.
\end{align*}
For each $\lambda\neq 0$, we define the complex-valued functions $\varphi_\lambda$ by
\begin{align}\label{DEF:varphi lambda}
\varphi_\lambda (t,x,v) := \varphi(t,x,v)e^{i\lambda(t-\tau_-(x,v))}. 
\end{align}
Notice that the choice of the phase function $t-\tau_-(x,v)$ is different from the Euclidean case, since in general the inner product $x\cdot v$ does not make sense on Riemannian manifolds.
Clearly, $\varphi_\lambda$ are also solutions to $\p_t f+X f=0$. 

We next define the function
$$
    \Theta_\sigma(x,v) := e^{-\int^{\tau_-(x,v)}_0\sigma(\gamma_{x,v}(-\tau_-(x,v)+s),\dot{\gamma}_{x,v}(-\tau_-(x,v)+s))\,ds},
$$
which satisfies $(\p_t + X + \sigma) \Theta_{\sigma}(x,v)=0$. To see this, notice that the function $\sigma(\rho_{x,v}(-\tau_-(x,v)+s))$  
is also invariant along the geodesic flow, i.e.
$$X \sigma(\rho_{x,v}(-\tau_-(x,v)+s))=0,$$
thus applying $X\tau_-(x,v)=1$ again gives 
\begin{align*}
    X \Theta_\sigma(x,v) = -\sigma(x,v)\Theta_\sigma(x,v).
\end{align*}
Summing up, the function
$
\varphi_\lambda \Theta_\sigma (t,x,v) 
$
is a solution to the scattering-free linear transport equation $(\p_t + X + \sigma)f=0$ in $\R\times SM$.

When $\overline M$ is a two-dimensional compact non-trapping  Riemannian manifold with smooth strictly convex boundary, it is well-known that there are global isothermal coordinates $(x_1,x_2)$ on $M$ so that in these coordinates the metric $g$ has the form 
$$g_{jk}(x)=e^{2c(x)} \delta_{jk},$$
where $c=c(x)$ is a smooth real-valued function in $\overline{M}$ and $\delta_{jk}$ is the Kronecker delta function, see \cite[Chapter 3]{PSUbook} for reference. Under the isothermal coordinates, the unit tangent bundle $SM$ has local coordinates $(x_1, x_2, \theta)$
so that $v\in S_xM$ is given by 
$$v(x,\theta)=e^{-c(x)}\LC\cos\theta \frac{\p}{\p x_1}+\sin\theta \frac{\p}{\p x_2}\RC,$$
where $\theta$ is the angle between a unit vector $v$ and ${\p\over\p x_1}$. The vertical vector field $V: C^\infty(SM)\to C^\infty(SM)$ is defined by 
$$
    Vu(x,\theta) :=\frac{\p}{\p \theta}u(x,\theta).
$$

Let $\Phi(t,x,v,v') := \mu(x,v,v') \varphi(t,x,v')\Theta_\sigma(x,v')$.
We denote a set $\mathcal{M}$ by 
\begin{equation}\label{set:mu}
\begin{split}
    \mathcal{M} :=\Big\{(\sigma,\mu)\in\Omega 
    :\  \mu(x,v,v')&=\mu(x,v',v)\,
    \hbox{ and $V\LC{\Phi(t,x,v,\cdot)\over V\tau_-(x,\cdot)}\RC \in L^2(S_xM)$}\quad \\
    &\hbox{for all }\phi \in C^\infty_0(\R,C^\infty_0(\p_- SM))  \hbox{ and for a.e. } x\in M\Big\}.    
\end{split}
\end{equation}

\begin{proposition}[GO solutions in Riemannian setting]\label{prop:GO M}
Let $(M,g)$ be the interior of a compact non-trapping Riemannian manifold with smooth strictly convex boundary $\p M$, of dimension $dim\,M=2$. 
Let $(\sigma,\mu)\in \mathcal{M}$. For any $\lambda\neq 0$ and $\phi\in C^\infty_0(\R,C^\infty_0(\p_- SM))$, the linear transport equation 
$$
    \p_t u+X u +\sigma u= K(u)
    $$
has solutions of the form 
\begin{align}\label{M GO positive}
	u_\lambda (t,x,v) &= \varphi_\lambda (t,x,v) \Theta_\sigma (x,v) + r_\lambda(t,x,v)
\end{align}
in the space $L^\infty(SM_T)$.
Moreover,  
the remainder term $r_\lambda$ satisfies the estimate 
\begin{align}\label{M bound r lambda}
    \|r_\lambda\|_{L^\infty(SM_T)}\leq C \sigma_0\|\phi\|_{L^\infty(\R\times\p_-SM)}.
\end{align}

In particular, 
the remainder $r_\lambda$ satisfies the following decay property:
\begin{align}\label{M limit r positive}
\lim\limits_{\lambda\rightarrow\pm\infty} \|r_\lambda\|_{L^2(SM_T) } =0.
\end{align}
Here the constant $C>0$ depends only on $\sigma$ and $T$. 
\end{proposition}
\begin{proof}
Following the same argument as the proof of Proposition~\ref{prop:GO}, the estimate \eqref{M bound r lambda} is valid for the manifold $M$ as well. Therefore, we will only focus on showing \eqref{M limit r positive}.

To this end, by \eqref{stability Lp} with $p=2$, the remainder $r_\lambda$ also satisfies 
\begin{align}\label{Decay r}
    \|r_\lambda\|_{L^2(SM_T)} 
    \leq C\|K(\varphi_\lambda \Theta_\sigma)\|_{L^2(SM_T)}.
\end{align}
To analyze the decay property of $r_\lambda$ in $L^2$ norm in the manifold $M$, it is sufficient to study the operator $K$. We write $K(\varphi_\lambda \Theta_\sigma)(t,x,v)=e^{ i\lambda t}\mathcal{K}^M_{\Phi,\lambda}(t,x,v)$,  
where $\mathcal{K}^M_{\Phi,\lambda}$ is defined by
$$
    \mathcal{K}_{\Phi,\lambda}^M (t,x,v) : =\int_{S_xM}e^{ -i\lambda \tau_-(x,v')} \Phi(t,x,v,v') \,dv' 
$$
with 
$$
    \Phi(t,x,v,v') := \mu(x,v,v') \varphi(t,x,v')e^{-\int^{\tau_-(x,v')}_{0} \sigma(\gamma_{x,v'}(-\tau_-(x,v')+s),\dot{\gamma}_{x,v'}(-\tau_-(x,v')+s))ds}.
$$
Note that the volume form $dv'$ of $(S_xM, g(x))$ is $d\theta$ \cite{PSUbook}. With this, applying the integration by parts, we derive
\begin{align*}
    \int_{S_xM} e^{-i\lambda \tau_-(x,v')}\Phi(t,x,v,v')\,dv' & =\int_0^{2\pi} e^{-i\lambda \tau_-(x,v'(x,\theta))}\Phi(t,x,v,v'(x,\theta))\,d\theta\\
    & =\int_0^{2\pi} \frac{1}{-i\lambda \p_\theta\tau_-} \Phi\,d(e^{-i\lambda \tau_-})\\
    &=\frac{1}{i\lambda}\int_0^{2\pi} e^{-i\lambda \tau_-}\p_\theta \LC \frac{\Phi}{\p_\theta \tau_-}\RC\,d\theta\\
    & = \frac{1}{i\lambda}\int_{S_xM} e^{-i\lambda \tau_-(x,v')} V\LC\frac{\Phi}{V\tau_-}\RC(t,x,v,v')\,dv'.
\end{align*} 
Due to $(\sigma,\mu)\in \mathcal{M}$, we know $V(\Phi/V\tau_-)\in L^2(S_xM)$.  
Thus $\mathcal{K}_{\Phi,\lambda}^M\to 0$ as $\lambda\to\pm\infty$, which gives $\|\mathcal{K}_{\Phi,\lambda}^M\|_{L^2(SM_T)}\to 0$ by applying Lebesgue dominate convergence theorem. Also with \eqref{Decay r} and $\|K(\varphi_\lambda \Theta_\sigma)\|_{L^2(SM_T)} = \|\mathcal{K}_{\Phi,\lambda}^M\|_{L^2(SM_T)}$, we obtain \eqref{M limit r positive}.
\end{proof}

\begin{remark}
The assumption on the decay of $V(\Phi/V\tau_-)$ is necessary. For example, let $M$ be the Euclidean disk $D=\{x\in\mathbb R^2: |x|<1\}$ with the polar coordinate $(r,\eta)$ with $0\leq r<1$ and $0\leq \eta < 2\pi$. Then for $(x,v)\in SM$, $x=(r\cos\eta, r\sin\eta)$, and let $v=(\cos\theta,\sin\theta)$. Given $(x,v)\in SM$, one can check that
$$
    \tau_-(x,v)=\tau_-(r,\eta,\theta)=r\cos(\theta-\eta)+\sqrt{1-r^2\sin^2(\theta-\eta)}, 
$$ 
thus
$$
    \p_\theta \tau_-=-r\sin(\theta-\eta)\LC 1+\frac{r\cos(\theta-\eta)}{\sqrt{1-r^2\sin^2(\theta-\eta)}}\RC.
$$
It follows that
$$
    |\p_\theta \tau_-|\leq (r+r^2) |\sin (\theta-\eta)| 
$$
since 
$$
    \frac{r\cos(\theta-\eta)}{\sqrt{1-r^2\sin^2(\theta-\eta)}}\leq \frac{r\cos(\theta-\eta)}{\sqrt{\cos^2(\theta-\eta)+ (1-r^2)\sin^2(\theta-\eta)}}\leq r.
$$
Notice that when $\theta$ is sufficiently close to $\eta$, $|\frac{1}{\sin(\theta-\eta)}|\sim \frac{1}{|\theta-\eta|}$, which is not in $L^1$ or $L^2$ spaces.
\end{remark}

\begin{remark}\label{examples}
    One may replace the assumption $V(\Phi/V\tau_-)\in L^2(S_xM)$ by the hypothesis that $\sigma, \mu$ are sufficiently regular, and for any $(x,v)\in SM$, $\mu(x, v,\cdot)=0$ near $v'\in S_xM$ such that $V\tau_-(x,v')=0$.
\end{remark}
\begin{remark}
    Though the constructions of GO solutions for $\R^n$ ($n\geq 2$) and a two-dimensional manifold $M$ have the same format and share alike properties. However, unlike the Euclidean space, the $x$ and $v$ variables are related in the Riemanian manifold so that certain constraints are needed to be imposed on $\mu$ to ensure similar decay property of the remainder term in the construction of GO solutions.
\end{remark}

Similarly, the conjugate equation also has the corresponding GO solutions for the Geometric setting.
\begin{proposition}\label{prop:conjugate GO M}
Suppose all hypothesis in Proposition~\ref{prop:GO M} are satisfied.
For any $\lambda\neq 0$ and $\phi\in C^\infty_0(\R,C^\infty_0(\p_- SM))$, the linear transport equation 
$$
    \p_t u+X u  -\sigma u= -K^*(u)
$$
has solutions of the form 
\begin{align}\label{GO negative M}
	u_\lambda (t,x,v) &= \varphi_\lambda (t,x,v) \Theta_{-\sigma} (t,x,v) + r_\lambda(t,x,v)
\end{align}
in the space $L^\infty(SM_T)$.
Here $K^*(f) (t,x,v):= \int_{S_xM} \mu(x,v',v) f(t,x,v')\,dv'$. 
Moreover,  
the remainder term $r_\lambda$ satisfies the estimate
$$
    \|r_\lambda\|_{L^\infty(SM_T)}\leq C \sigma_0\|\phi\|_{L^\infty(\R\times\p_-SM)}
$$
and has the decay property
\begin{align}\label{limit r negative M}
\lim\limits_{\lambda\rightarrow\pm\infty} \|r_\lambda\|_{L^2(SM_T) } =0.
\end{align}
Here the constant $C>0$ depends on $\sigma$ and $T$. 
\end{proposition}
We refer to \cite[section~3]{bellassouedInverseProblemLinear2019} for the construction of GO solutions in a different space.

\begin{remark}\label{remark:comparison}  
As mentioned in the introduction, our construction of the GO solutions for both Euclidean and Riemannian settings is different from the one in \cite{bellassouedInverseProblemLinear2019}. The key feature is that we do not need to embed the manifold $M$ into $\mathbb R^n$ or a larger manifold, which makes our construction more suitable for the general Riemannian case.

Let us discuss the difference of GO solutions in more details. Assume that $M$ is a smooth open and bounded convex domain in $\mathbb R^n$, so $SM=M\times \mathbb S^{n-1}$. Both approaches result in a solution to the linear transport equation $\p_t f+Xf+\sigma f=K(f)$ of the same form
$$u_\lambda=\varphi \,e^{i\lambda p(t,x,v)}\,\Theta_\sigma+r_\lambda,$$
where $\varphi$ and $p$ solve the transport equation $\p_t f+Xf=0$, and $\Theta_\sigma$ solves the transport equation $\p_t f+Xf+\sigma f=0$, while $r_\lambda$ being the remainder term.

The difference between our construction and \cite{bellassouedInverseProblemLinear2019} mainly arises from the choice of $\varphi$ and $\Theta_\sigma$ in the leading term. In \cite{bellassouedInverseProblemLinear2019} they are constructed in the whole space by taking  
$$
    \varphi(t,x,v)=\phi(x-tv,v)\quad \mbox{with} \quad \phi\in C^\infty_0(\mathbb R^n; C(\mathbb S^{n-1}))
$$
and 
$$
    \Theta_\sigma(t,x,v) = e^{-\int^t_0\sigma(x-(t-s)v,v)ds}.
$$ 
Notice that for $(t,x,v)\in SM_T$, it is possible that $x-tv\notin M$ and therefore it is necessary to embed the manifold $M$ in a larger domain. However, in our approach, since all the constructions are restricted in the manifold $\overline M$, there is no need of such extensions. In particular, the choice of $\varphi_\lambda$ in \eqref{DEF:varphi lambda} suits the nature of manifold better.
\end{remark}

\section{Inverse coefficient problems}\label{sec:ICP}
In this section, the focus will be recovering the nonlinear coefficient by utilizing GO solutions, constructed above in the Euclidean space and in the two-dimensional Riemannian manifold $(M,g)$. All results here will hold for both Euclidean and Riemannian settings, unless otherwise specified.

Before deriving the crucial integral identity, we first impose small parameters into the data in the initial boundary value problem \eqref{EQN: RTE equ} and perform the linearization scheme.

\subsection{Linearization of the initial boundary value problem}\label{sec:linearization}
In this section, we apply the linearization method to decompose the solution of the nonlinear transport equation.
We fix $(h_0,h_-)\in\mathcal{X}_{\delta}$ defined in \eqref{DEF:set X} for small $\delta>0$. For $\varepsilon \in (0,1)$, based on the well-posedness in Theorem~\ref{THM:WELL}, there exists a unique solution $f_\varepsilon\equiv f_{\varepsilon}(t,x,v)\in L^\infty(SM_T)$ to the problem:
 \begin{align}\label{EQN: nonlinear equ}
 	\left\{\begin{array}{rcll}
 		\p_tf_\varepsilon+ X f_\varepsilon+ \sigma f_\varepsilon +q  f_\varepsilon^m &=& K(f_\varepsilon) & \hbox{in } SM_T,  \\
 		f_\varepsilon &=&\varepsilon h_0& \hbox{on } \{0\}\times SM,\\
 		f_\varepsilon &=&\varepsilon h_- & \hbox{on }\p_-SM_T,
 	\end{array}\right.
 \end{align}
with the estimate
\begin{align}\label{EST:nonlinear f}
    \|f_\varepsilon\|_{L^\infty(SM_T)}\leq C \varepsilon\LC \|h_0\|_{L^\infty(SM)}+\|h_-\|_{L^\infty(\p_-SM_T)} \RC.
\end{align}
Such solution $f_\varepsilon$ can be decomposed as follows.  
\begin{proposition}
There exists a unique solution $f_\varepsilon$ to \eqref{EQN: nonlinear equ}, which can be expanded in the following form:
\begin{align}\label{f expansion}
    f_\varepsilon = \varepsilon u + \varepsilon^m w + R_\varepsilon,
\end{align}
where $m\geq 2$ and $u\in L^\infty(SM_T)$ is the solution to 
\begin{align}\label{EQN: 1 linear}
	\left\{\begin{array}{rcll}
		\p_t u+ X u + \sigma u &=& K (u)& \hbox{in } SM_T,  \\
		u &=&h_0 & \hbox{on } \{0\}\times SM,\\
		u &=&h_- & \hbox{on }\p_-SM_T,
	\end{array}\right.
\end{align}
and $w\in L^\infty(SM_T)$ is the solution to 
\begin{align}\label{EQN: k linearization}
	\left\{\begin{array}{rcll}
		\p_t w + X w + \sigma w + qu^m&=& K (w) & \hbox{in } SM_T,  \\
	    w &=&0 & \hbox{on } \{0\}\times SM,\\
		w &=&0 & \hbox{on }\p_-SM_T.
	\end{array}\right.
\end{align}

Moreover, the remainder term $R_\varepsilon$ solves the problem:
 \begin{align}\label{EQN: linear R}
	\left\{\begin{array}{rcll}
		\p_t R_\varepsilon+ X R_\varepsilon+ \sigma R_\varepsilon  &=& K(R_\varepsilon) - qf^m_\varepsilon+ \varepsilon^m qu^m& \hbox{in } SM_T,  \\
		R_\varepsilon &=& 0 & \hbox{on } \{0\}\times SM,\\
		R_\varepsilon &=& 0 & \hbox{on }\p_-SM_T,
	\end{array}\right.
\end{align}
and satisfies
\begin{align}\label{EST:R}
	 \|R_\varepsilon\|_{L^\infty(SM_T)}\leq C\varepsilon^{2m-1}\LC \|h_0\|_{L^\infty(SM)} + \|h_-\|_{L^\infty(\p_-SM_T)}\RC^{2m-1},
\end{align}
where the constant $C>0$ depends on $\sigma,\,q$ and $T$.
\end{proposition}
\begin{proof}
The unique existence of solution $u\in L^\infty(SM_T)$ to \eqref{EQN: 1 linear} is due to Proposition~\ref{prop:forward problem linear}. Since one can view $qu^m$ as a source term in the first equation in \eqref{EQN: k linearization}, there exists a unique solution $w\in L^\infty(SM_T)$ for the problem \eqref{EQN: k linearization}. Let $F=f_\varepsilon-\varepsilon u$, which then solves
	\begin{align}\label{EQN: linear F}
		\left\{\begin{array}{rcll}
			\p_t F+ X F + \sigma F &=& K (F) - qf^m_\varepsilon& \hbox{in } SM_T,  \\
			F &=&0 & \hbox{on } \{0\}\times SM,\\
			F &=&0 & \hbox{on }\p_-SM_T.
		\end{array}\right.
	\end{align}
By Proposition~\ref{prop:forward problem linear} and \eqref{EST:nonlinear f}, one has
\begin{align}\label{EST: diff F}
	\|F\|_{L^\infty(SM_T)}\leq C \|q\|_{L^\infty(SM_T)} \|f_\varepsilon\|^m_{L^\infty(SM_T)}\leq C \|q\|_{L^\infty(SM_T)} \varepsilon^m\LC \|h_0\|_{L^\infty(SM)} + \|h_-\|_{L^\infty(\p_-SM_T)}\RC^m.
\end{align}

It remains to show the existence of $R_\varepsilon$. From the hypothesis, since the source term $- qf^m_\varepsilon+ \varepsilon^m qu^m$ is bounded, applying Proposition~\ref{prop:forward problem linear} yields the unique existence of $R_\varepsilon$ to the linear transport equation \eqref{EQN: linear R} and also
\begin{align*}
	\|R_\varepsilon\|_{L^\infty(SM_T)}
	&\leq C\|- qf^m_\varepsilon+ \varepsilon^m qu^m\|_{L^\infty(SM_T)}\\
	&\leq C \|q\|_{L^\infty(SM)}\left\| (f_\varepsilon - \varepsilon u)\sum^{m-1}_{k=0} f_\varepsilon^{m-1-k} (\varepsilon u)^k\right\|_{L^\infty(SM_T)}\\
        &=C \|q\|_{L^\infty(SM)}\left\| F\right\|_{L^\infty(SM_T)}\left\| \sum^{m-1}_{k=0} f_\varepsilon^{m-1-k} (\varepsilon u)^k\right\|_{L^\infty(SM_T)}\\
	&\leq C\varepsilon^{2m-1}\LC \|h_0\|_{L^\infty(SM)} + \|h_-\|_{L^\infty(\p_-SM_T)}\RC^{2m-1}.
\end{align*}
Here we have used \eqref{EST:nonlinear f} and \eqref{EST: diff F} in the last inequality above.
\end{proof}
 
\subsection{Integral identity}
Next, we derive an integral identity that relates the unknown coefficient $q$ with the measurements provided that the coefficients $\sigma$ and $\mu$ are given. 

To this end, we view the function $f_\varepsilon$ as a function of $\varepsilon$. Then following the notations above, we define the $\ell$-th order finite differences operators at $0$ by 
$$
    \Delta^\ell_\varepsilon[f_\varepsilon] := \varepsilon^{-\ell}\sum^\ell_{k=0} (-1)^{\ell-k}
    \LC\begin{array}{c} 
    	\ell\\k\\
    \end{array}\RC   f_{k\varepsilon},\quad \ell\geq 1,
$$
whose limit satisfies 
$$
    \p_\varepsilon^\ell f_\varepsilon|_{\varepsilon=0} = \lim_{\varepsilon\rightarrow 0}\Delta^\ell_\varepsilon [f_\varepsilon]  .
$$
For instance, when $\ell=2$, 
$$
    \Delta^2_\varepsilon[f_\varepsilon] = \varepsilon^{-2}  (f_{2\varepsilon}-2f_{\varepsilon}),
$$
where we used the fact that the solution $f_\varepsilon\equiv 0$ when $\varepsilon=0$ according to the well-posedness theorem.
Let the finite difference operator acting on \eqref{f expansion} gives  
\begin{align*} 
    \Delta^m_\varepsilon [f_\varepsilon] = (m!) w + \Delta^m_\varepsilon [R_\varepsilon].
\end{align*}
Hence, together with \eqref{EST:R}, as $\varepsilon\rightarrow 0$, we obtain
$$
     \p_\varepsilon^m f_\varepsilon|_{\varepsilon=0} =(m!) w
$$	
as well as the measurement operator
\begin{align}\label{m diff data}
    \p^m_\varepsilon \mathcal{A}_q(\varepsilon h_0,\varepsilon h_-)|_{\varepsilon=0} =(m!)   (w|_{t=T}, w|_{\p_+SM_T}).
\end{align}
Indeed, $\p_\varepsilon f_\varepsilon|_{\varepsilon=0} = u$, the solution to the linear transport equation. Similar finite difference method was also applied to solve inverse problems for nonlinear equations, see for instance, \cite{LLPT2020_1} for the semilinear wave equation, \cite{LLZ2023} for the nonlinear Schr\"odinger equation, and \cite{LY2023} for the Boltzmann equation.

\begin{proposition}[Integral identity]\label{prop:id}
Let $f_j = \varepsilon u_j+\varepsilon^m w_j+ R_{\varepsilon,j}$ be the solution to \eqref{EQN: nonlinear equ}, 
where $u_j$, $w_j$ and $R_{\varepsilon,j}$ solve \eqref{EQN: 1 linear}, \eqref{EQN: k linearization} and \eqref{EQN: linear R}, respectively, with $q=q_j\in L^\infty(SM_T)$ for $j=1,\,2$. Let $\mathcal{A}_{q_1}$ and $\mathcal{A}_{q_2}$ be the corresponding albedo operator of \eqref{EQN: nonlinear equ}. 
If $\mathcal A_{q_1}(h_0,h_-)=\mathcal A_{q_2}(h_0,h_-)$ for all $(h_0,h_-)\in \mathcal{X}_{\delta}$, then  
\begin{align}\label{Integral identity}
     \int_{SM_T}(q_1-q_2) u^m \overline{u}_0 \,dtd\Sigma = 0,
\end{align} 
where $u\in L^\infty(SM_T)$ is the solution to \eqref{EQN: 1 linear} and $u_0$ is the solution to its conjugate equation 
\begin{align*} 
    \p_tu_0 + X u_0 - \sigma u_0 + K^*u_0 =0 \quad \hbox{ in } SM_T. 
\end{align*} 
\end{proposition}
\begin{proof}
    Since $(\sigma, \,\mu)$ are given and the equation \eqref{EQN: 1 linear} 
is independent of $q_j$, by the uniqueness of the solution to the linear transport equation that 
$$
    u:=u_1=u_2.
$$

	Let $\widetilde{w}=w_1-w_2$. Then $\widetilde w$ is the solution to 
		\begin{align}\label{EQN: 2 linear w}
		\left\{\begin{array}{rcll}
			\p_t \widetilde{w}   + X \widetilde{w} + \sigma \widetilde{w} -K\widetilde{w} &=&  -(q_1-q_2) u^m& \hbox{in } SM_T,  \\
			\widetilde{w} &=& 0 & \hbox{on } \{0\}\times SM,\\
			\widetilde{w} &=& 0 & \hbox{on }\p_-SM_T.
		\end{array}\right.
		\end{align}		 
    Due to the fact that $\widetilde{w}|_{t=0}=0$ and $\widetilde{w}|_{\p_-SM_T}=0$, the identity
    \begin{align}\label{Integral identity w}
        - \int_{SM_T}(q_1-q_2) u^m \overline{u}_0 \,dtd\Sigma = \int_{\p_+SM_T} \widetilde w\overline{u}_0 \,d\xi dt +\int_{SM} \widetilde w\overline{u}_0(T,x,v)\,d\Sigma,
    \end{align}
    follows by multiplying the first equation in \eqref{EQN: 2 linear w} by $\overline{u}_0$ and applying the integration by parts. Since the albedo operators are the same, by \eqref{m diff data}, we get $\widetilde{w}|_{t=T} =\widetilde{w}|_{\p_+SM_T}=0$ such that the right hand side of \eqref{Integral identity w} vanishes. This results in the desired identity \eqref{Integral identity}.
\end{proof} 
 
\begin{remark} 
We would like emphasize that our albedo operator $\mathcal{A}$ consists of both initial/final data and incoming/outgoing data. This indicates the right-hand side of \eqref{Integral identity} is known and is equal to zero when $\mathcal{A}_{q_1}=\mathcal{A}_{q_2}$. In particular, thanks to the information of the initial and final data, the solutions $u$ and $u_0$ 
above do not necessarily satisfy $u(0,\cdot,\cdot)=0$ and $u_0(T,\cdot,\cdot)=0$. As a result, freedom of choice is given to select the GO solutions below. 
\end{remark}

In the remaining part of this section, we will take the solution $u$ of the linear transport equation \eqref{EQN: 1 linear} to be the GO solutions of the form \eqref{GO positive} in the Euclidean space, $n\geq 2$ (or \eqref{M GO positive} in a two-dimensional manifold):
\begin{align}\label{u_12}
    u(t,x,v)= \varphi_{\lambda} (t,x,v) \Theta_{\sigma} (x,v) + r_{\lambda} (t,x,v),\quad \lambda \neq 0.
\end{align}
Let $h_0=u|_{t=0}\in L^\infty (SM)$ and $h_-=u|_{\p_-SM_T}\in L^\infty(\p_-SM_T)$ be the initial and boundary data, respectively.
We also take the GO solutions $u_0$ of the conjugate equation $\p_tu_0 + X u_0- \sigma u_0 + K^*u_0 =0$, which is of the form 
\begin{align}\label{u_-}
    u_0(t,x,v)= \varphi_{\eta} (t,x,v) \Theta_{-\sigma} (x,v) + r_{\eta}(t,x,v), \quad \eta\neq 0
\end{align}
based on \eqref{GO negative} (or \eqref{GO negative M}). 
Hence, we can further derive the following identity from \eqref{Integral identity}, which is valid in both the Euclidean and Riemannian settings. 

\begin{lemma}\label{lemma:GO id} Suppose that all conditions in Proposition~\ref{prop:GO}-\ref{prop:conjugate GO M} and Proposition~\ref{prop:id} hold. If $\mathcal A_{q_1} =\mathcal A_{q_2}$ on $\mathcal{X}_{\delta}$, then 
\begin{align}\label{ID: GO} 
	\lim\limits_{\lambda,\eta\rightarrow \infty}\int_{SM_T} (q_1-q_2) \varphi_\lambda^m \overline{\varphi}_\eta  \Theta_\sigma^{m-1} \,dtd\Sigma =0.
\end{align}
\end{lemma}
\begin{proof} 
    It is sufficient to show the case $m=2$ since $m>2$ can be justified similarly. To this end, substituting $u$ and $u_0$ of the form \eqref{u_12} and \eqref{u_-} into the identity \eqref{Integral identity} with $m=2$ gives that
    \begin{align*} 
      \int_{SM_T}(q_1-q_2)(\varphi_\lambda)^2\overline{\varphi}_\eta \Theta_\sigma \,dtd\Sigma + I_{\lambda,\eta}= 0.
    \end{align*}
     Denote
    $$
       \widetilde{q}:=q_1-q_2.
    $$
    Here $I_{\lambda,\eta}$ includes all the higher order terms and is defined as follows:
    \begin{align}\label{I_lambda}
        I_{\lambda,\eta}:=  \int_{SM_T} \widetilde{q} \left((\varphi_\lambda)^2   (\Theta_\sigma)^2 \overline{r}_\eta + 2\varphi_\lambda \overline{\varphi}_\eta r_\lambda + 2\varphi_\lambda \Theta_\sigma r_\lambda \overline{r}_\eta+ \overline{\varphi}_\eta \Theta_{-\sigma}(r_\lambda)^2 + (r_\lambda)^2\overline{r}_\eta \right) \,dtd\Sigma.
    \end{align}
    We claim that   
    $$
    \lim\limits_{\lambda,\eta\rightarrow \infty} I_{\lambda,\eta}=0.
    $$
    To see this, we split the discussion into three parts. First we consider the first two terms in \eqref{I_lambda}. 
    Note that $|\Theta_\sigma|\leq e^{\sigma_0\, diam(M)}$ since $(\sigma,\mu)\in \Omega$. Here $diam(M)$ denotes the diameter of $M$.  
    By applying the H\"older's inequality, the first term of $I_{\lambda,\eta}$ then satisfies
    $$
    \LV\int_{SM_T}  \widetilde{q} (\varphi_\lambda)^2   (\Theta_\sigma)^2 \overline{r}_\eta \,dtd\Sigma \RV \leq C\|\widetilde{q}\|_{L^\infty(M_T)} \|\varphi\|_{L^\infty(SM_T)}\|\varphi\|_{L^2(SM_T)} \|r_\eta\|_{L^2(SM_T)},
    $$
    which goes to zero when $\eta\rightarrow\infty$ by applying \eqref{limit r negative} (or \eqref{limit r negative M} in the manifold),
    where the constant $C>0$ depends on $M$ and $\sigma_0$. Similarly, the second term also goes to zero. 

    Next we deal with the third and fourth terms. Applying H\"older's inequality yields that
    $$
    \LV\int_{SM_T} \widetilde{q} \varphi_\lambda \Theta_\sigma r_\lambda \overline{r}_\eta  \,dtd\Sigma \RV \leq C\|\widetilde{q}\|_{L^\infty(M_T)} \|\varphi\|_{L^\infty(SM_T)}\|r_\lambda\|_{L^2(SM_T)}\|r_\eta\|_{L^2(SM_T)},
    $$
    which goes to zero as $\lambda,\eta\rightarrow \infty$ by \eqref{limit r positive} and \eqref{limit r negative}(or \eqref{M limit r positive}, \eqref{limit r negative M} in the manifold). The fourth term follows by using a similar argument as above.

    Finally, we apply H\"older's inequality and the estimates 
    \eqref{bound r lambda}, \eqref{limit r positive}, \eqref{limit r negative}  (or  \eqref{M bound r lambda}, \eqref{M limit r positive}, \eqref{limit r negative M} in the manifold) to deduce that
    \begin{align*}
        \LV\int_{SM_T}\widetilde{q} (r_\lambda)^2 \overline{r}_\eta \,dtd\Sigma\RV \leq C\|\widetilde{q}\|_{L^\infty(M_T)} 
        \|r_\lambda\|_{L^2(SM_T)}\|r_\eta\|_{L^2(SM_T)}
    \end{align*}
    converges to zero as $\lambda,\,\eta\rightarrow \infty$. This completes the proof of the lemma.
\end{proof}

\subsection{Reduction to the light ray transform}  
With suitably chosen functions $\varphi_\lambda$ and $\varphi_\eta$ in the GO solutions, the identity \eqref{ID: GO} can be reduced to a weighted light ray transform.

\subsubsection{Geometric setting}
We will bring down the integral identity \eqref{ID: GO} with the following special maps.
\begin{lemma}
For almost every $(x_0,v_0)\in\p_-SM$, there is a family of maps $P_{\kappa,x_0,v_0}\in C^\infty_0(\p_-SM)$  with $0<\kappa \ll 1$ such that
$$\|P_{\kappa,x_0,v_0}\|_{L^1(\p_-SM)}=1,$$
and, moreover, for any given $f\in L^\infty(\p_-SM)$, the following holds:
\begin{equation}\label{limit of P}
\begin{split}
    \lim_{\kappa\to 0}\int_{\p_-SM} P_{\kappa,x_0,v_0}(x,v) f(x,v)\,d\xi =f(x_0,v_0).
   \end{split}
\end{equation}
\end{lemma}
\begin{proof}
Given $(x_0,v_0)\in \p_-SM$, let $(x,v): U\times W\subset \mathbb R^{n-1}\times\mathbb R^{n-1} \to \p SM$ be a coordinate chart near $(x(0),v(0))=(x_0,v_0)$. Let $d \xi(x,v)=|\langle v,\nu(x)\rangle|\sqrt{\det g_-(x(\textbf{u}),v(\textbf{u},\textbf{w}))}\,d\textbf{u}d\textbf{w}$ be the local coordinate expression, where $g_-$ is the metric on $\p SM$ induced by the Sasaki metric on $SM$ \cite{PSUbook}.
Let $\psi\in C^\infty_0(\mathbb R^{n-1};\mathbb R)$ satisfy that $\psi\geq 0$, supp $\psi\subset B_0(1)$ the unit open ball, and $\|\psi\|_{L^1(\mathbb R^{n-1})}=1$. 
For $\kappa>0$, we denote $\psi_\kappa(\textbf{u}):=\psi(\textbf{u}/\kappa)/\kappa^{n-1}$. Then $\psi_\kappa\geq 0$ and $\|\psi_\kappa\|_{L^1(\mathbb R^{n-1})}=1$. 
Let $0<\kappa\ll 1$, so that the open ball $B_0(\kappa)$ with center at $0$ and radius $\kappa$ satisfies $B_0(\kappa)\subset U\cap W$. We define $P_{\kappa,x_0,v_0}\in C^\infty_0(\p_-SM)$ by
\begin{align*}
    P_{\kappa,x_0,v_0}(x,v)=\left\{\begin{array}{cll}
		\frac{1}{|\langle v,\nu(x)\rangle| \sqrt{\det g_-(x,v)}}\psi_\kappa(\textbf{u}(x))\psi_\kappa(\textbf{w}(x,v)) &, & \hbox{if } (\textbf{u}(x),\textbf{w}(x,v))\in U\times W,  \\
			0 &, & \hbox{otherwise}.
		\end{array}\right.
\end{align*}

Given any $f\in L^\infty(\p_-SM)$,
\begin{align*}
    &\quad  \int_{\p_-SM} P_{\kappa,x_0,v_0}(x,v) f(x,v)\,d\xi(x,v)\\
   &=  \int_{U\times W}  
   \psi_\kappa(\textbf{u})\psi_\kappa(\textbf{w}) f(x(\textbf{u}),v(\textbf{u},\textbf{w})) 
   \,d\textbf{u}d\textbf{w} \\
   &=  \int_{\mathbb R^{n-1}\times\mathbb R^{n-1}} \psi_\kappa(\textbf{u})\psi_\kappa(\textbf{w}) f(x(\textbf{u}),v(\textbf{u},\textbf{w}))\,d\textbf{u}d\textbf{w}.
\end{align*}

In particular, since $\|\psi_\kappa\|_{L^1(\mathbb R^{n-1})}=1$ for any $\kappa>0$, by letting $f\equiv 1$, we get that $$\|P_{\kappa,x_0,v_0}\|_{L^1(\p_-SM)}=1.$$
Moreover, by \cite[Theorem 8.15]{Folland1984}, it follows that
\begin{equation*} 
\begin{split}
    & \quad\lim_{\kappa\to 0}\int_{\p_-SM} P_{\kappa,x_0,v_0}(x,v) f(x,v)\,d\xi\\
   &= \lim_{\kappa\to 0}\int_{\mathbb R^{n-1}\times\mathbb R^{n-1}} \psi_\kappa(\textbf{u})\psi_\kappa(\textbf{w}) f(x(\textbf{u}),v(\textbf{u},\textbf{w}))\,d\textbf{u}d\textbf{w}=f(x_0,v_0)
   \end{split}
\end{equation*}
for a.e. $(x_0,v_0)\in \p_-SM$.
\end{proof}

We are ready to derive the light ray transform.
\begin{proposition} \label{reduction to light ray}
Let the assumptions of Lemma~\ref{lemma:GO id} be fulfilled.
Let $q_j\in L^\infty(M_T)$ for $j=1,2$.  
If $\mathcal A_{q_1}(h_0,h_-)=\mathcal A_{q_2}(h_0,h_-)$ for all $(h_0,h_-)\in \mathcal{X}_{\delta}$, then for almost every $t\in\mathbb R$ and $(x,v)\in \p_-SM$, we have
\begin{align*}
\int_0^{\tau(x,v)} \widetilde{q}(t+s,\gamma_{x,v}(s)) e^{-\int^s_0 (m-1)\sigma(\gamma_{x,v}(\ell),\dot{\gamma}_{x,v}(\ell))\,d\ell}\,ds=0.
\end{align*} 
\end{proposition} 
\begin{proof} 
We recall the notation $\widetilde{q}:=q_1-q_2$ and extend $\widetilde{q}(\cdot,x)$ by zero to the whole space $\R$ in the $t$ variable.
Recall that the geodesic flow is denoted by 
$$
\rho_{x,v}(t):=(\gamma_{x,v}(t),\dot{\gamma}_{x,v}(t)).
$$
Given $(x_0,v_0)\in \p_-SM$, now we take $\eta=m\lambda$ and the functions 
\begin{align*}
    \varphi_\lambda (t,x,v) & =P^{1\over m}_{\kappa,x_0,v_0}(\rho_{x,v}(-\tau_-(x,v)))\phi_1(t-\tau_-(x,v),\rho_{x,v}(-\tau_-(x,v)))e^{i\lambda p(t,x,v) },\\
    \varphi_\eta (t,x,v) & =\varphi_{m\lambda}(t,x,v)=\phi_2(t-\tau_-(x,v),\rho_{x,v}(-\tau_-(x,v)))e^{i m\lambda p(t,x,v)}
\end{align*}
in \eqref{u_12}-\eqref{u_-}, where $\phi_j\in C^\infty_0(\R,C^\infty_0(\p_- SM))$ for $j=1,2$. 
Here 
$p(t,x,v) =  t-x\cdot v $ for the Euclidean setting and $p(t,x,v) =  t-\tau_-(x,v) $ for the Riemannian setting.
Notice that since $P_{\kappa,x_0,v_0}(\rho_{x,v}(-\tau_-(x,v)))$ is invariant along the geodesic flow, $XP_{\kappa,x_0,v_0}(\rho_{x,v}(-\tau_-(x,v)))=0$. This implies  $\varphi_\lambda$ and $\varphi_\eta$ are solutions to $\p_t f+Xf=0$.

Since $\widetilde{q}$ has support in $M_T$, the integral in \eqref{ID: GO} can be extended from $[0,T]$ to $\R$. 
Applying the Santalo's formula \cite{PSUbook} to \eqref{ID: GO} yields that 
\begin{align*}\label{EST:lower bound of q 1}	 
    0&=\int_{\R}\int_{SM} \widetilde{q}(t,x) \phi_1^m\phi_2(t-\tau_-(x,v),\rho_{x,v}(-\tau_-(x,v))) P_{\kappa,x_0,v_0}(\rho_{x,v}(-\tau_-(x,v))) \Theta_\sigma^{m-1} \, d\Sigma(x,v) dt\notag\\
    &=  \int_{\R}\int_{\p_-SM} \int_0^{\tau(y,w)} \widetilde{q}(t,\gamma_{y,w}(s)) \phi_1^m\phi_2(t-s,y,w) P_{\kappa,x_0,v_0}(y,w) e^{-\int^s_0 (m-1)\sigma(\rho_{y,w}(\ell))\,d\ell} \,ds d\xi(y,w) dt.
\end{align*}
Let $\kappa\to 0$, we apply the change of variable $t\mapsto t'=t-s$, by \eqref{limit of P} and Fubini's theorem, to get
\begin{equation}\label{ID: integral of q}
        0= \int_0^{\tau(x_0,v_0)} \int_{\R} \widetilde{q}(t'+s,\gamma_{x_0,v_0}(s))\phi_1^m\phi_2 (t',x_0,v_0) e^{-\int_0^s (m-1)\sigma(\rho_{x_0,v_0}(\ell)) \, d\ell}\,dt' ds.
\end{equation}

Finally, for any $t_0\in\mathbb R$, since $\phi_j$ are arbitrary smooth functions with compact support, we let  $\phi_1\equiv 1$ in $(t_0-1,t_0+1)\times \p_-SM$ and $\phi_2(t,x,v) = \chi({t-t_0\over \zeta})/\zeta$, $0<\zeta<1$, where $\chi(t)\in C^\infty_0(\R)$ satisfies $\chi(0)=1$ and $\|\chi\|_{L^1(\R)}=1$ with support in $|t|\leq 1$. By \cite[Theorem 8.15]{Folland1984}, when $\zeta$ goes to zero, we then derive from \eqref{ID: integral of q} that 
\begin{align*}
	\int_0^{\tau(x_0,v_0)} \widetilde{q}(t_0+s,\gamma_{x_0,v_0}(s)) e^{-\int^s_0 (m-1)\sigma(\gamma_{x_0,v_0}(\ell),\dot{\gamma}_{x_0,v_0}(\ell))\,d\ell} \,ds =0,
\end{align*}
for almost every $t_0$, provided that $q_j\in L^\infty(M_T)$.  
\end{proof}

Notice that the integral appearing in Proposition \ref{reduction to light ray} defines the (weighted) light ray transform of $\widetilde{q}$ on $\R\times M$. More precisely, let's first introduce the definition of the standard \textit{light ray transform} $L$ of a function $S\in L^\infty(\R\times M)$. It is defined by 
$$
    L S(t,x,v):=\int_0^{\tau(x,v)} S(t+s,\gamma_{x,v}(s))\,ds, \quad (t,x,v)\in \R\times S\overline{M}.
$$

Let $W\in L^\infty(SM)$, we define the general \textit{weighted light ray transform} by
\begin{align}\label{def:weighted light ray}
    L_W S(t,x,v):=\int_0^{\tau(x,v)}  S(t+s,\gamma_{x,v}(s)) W(\gamma_{x,v}(s),\dot{\gamma}_{x,v}(s))\,ds.
\end{align}
Therefore, the weight appearing in the light ray transform deduced from Proposition \ref{reduction to light ray} is defined as
$$W(\gamma_{x,v}(s),\dot{\gamma}_{x,v}(s)):=e^{-\int_0^s (m-1)\sigma(\gamma_{x,v}(r),\dot{\gamma}_{x,v}(r))\,dr}, \quad s\in (0,\tau(x,v)),\quad (x,v)\in \p_-SM.$$
In particular, $W\in C^\infty(SM)$ if $\sigma\in C^\infty(S\overline{M})$.

\subsection{Injectivity of the light ray transform}\label{sec:main theorem}
As a preparation for the proof of our main results in the next subsection, we establish the following injectivity of the light ray transform $L_W$, based on analyticity. In particular, it works for a general nonvanishing weight $W\in L^\infty(SM)$.

The approach has been used in earlier study of the light ray transform without weight \cite{FIKO2021, FIO2021}.
\begin{hypothesis}\label{hypo}
	Let $(M,g)$ be a compact connected Riemannian manifold with smooth boundary. Assume that the geodesic X-ray transform with a nonvanishing weight $W\in L^\infty(S M)$ is injective. Namely, for any $S\in L^\infty (M)$, if 
	$$
	I_WS(x,v):=\int_0^{\tau(x,v)} S(\gamma_{x,v}(s)) W(\gamma_{x,v}(s),\dot{\gamma}_{x,v}(s))\,ds=0 
	$$ 
	for a.e. $(x,v)\in \p_-SM$, then $S=0$.
\end{hypothesis}

On simple manifolds, injectivity of $I_W$ has been proven for generic weights \cite{FSU}, including the real-analytic ones. On the other hand, injectivity results of $I_W$ are known on manifolds admitting strictly convex functions \cite{PSUZ2019}. See \cite{ZhouSurvey2024} for additional references.

\begin{theorem}\label{injectivity of weighted light ray}
Suppose that Hypothesis \ref{hypo} holds. Assume $S\in L^\infty(\mathbb R\times M)$ is supported in $M_T$. If $L_W S(t,x,v)=0$ for almost every $(t,x,v)\in \R\times \p_-SM$ with a nonvanishing weight function $W\in L^\infty(SM)$, 
then $S=0$.
\end{theorem}

\begin{proof} 
Consider 
$$
L_{W} S(t,x,v):=\int_0^{\tau(x,v)} S(t+s,\gamma_{x,v}(s))W(\gamma_{x,v}(s),\dot{\gamma}_{x,v}(s))\,ds,
\quad \hbox{for $a.e.$ } (t,x,v)\in \R\times \p_-SM.
$$

Taking the Fourier transform w.r.t. the time variable $t$, we get
\begin{align*}
	\widehat{L_{W}S} (\eta,x,v)
	&=\int_\R e^{-i\eta t}\int_0^{\tau(x,v)} S(t+s,\gamma_{x,v}(s))W(\gamma_{x,v}(s),\dot{\gamma}_{x,v}(s))\,dsdt.
\end{align*}
By the change of variable $t'=t+s$ and the Fubini's theorem, note that $S\in L^\infty(\mathbb R\times M)$, we derive
\begin{align*}
	\widehat{L_{W}S} (\eta,x,v)
	&= \int_0^{\tau(x,v)} e^{ i\eta s}\hat{S}(\eta,\gamma_{x,v}(s))W(\gamma_{x,v}(s),\dot{\gamma}_{x,v}(s))\,ds.
\end{align*}
 
When $\eta=0$,
$$
    \widehat{L_{W}S}(0,x,v) = \int_0^{\tau(x,v)}\hat{S}(0,\gamma_{x,v}(s))W(\gamma_{x,v}(s),\dot{\gamma}_{x,v}(s))\,ds =0\quad \hbox{for $a.e.$ } (x,v)\in \p_-SM
$$
for all geodesics $\gamma$ in $M$.  
Therefore, Hypothesis \ref{hypo} can be applied directly and it implies $\hat{S}(0,\cdot)=0$ in $M$.  Since the weight $W$ is independent of $\eta$, the derivatives of $\widehat{L_W S}$ at $\eta=0$ can be written as
\begin{align}\label{ID:higher deri}
   0= \p^k_\eta\widehat{L_{W}S}(\eta,x,v)|_{\eta=0} = \sum^k_{j=0}\left( \begin{array}{c}
   	k\\
   	j\\
   \end{array}\right)
\int_0^{\tau(x,v)}(is)^{k-j}\p^j_\eta\hat{S}(0,\gamma_{x,v}(s))W(\gamma_{x,v}(s),\dot{\gamma}_{x,v}(s))\,ds.
\end{align}
We then proceed by applying the induction argument. Suppose that $\p^j_\eta\hat{S}(0,\cdot)=0$ for $j<k$. From \eqref{ID:higher deri}, we obtain
$$
    0= \int_0^{\tau(x,v)} \p^k_\eta\hat{S}(0,\gamma_{x,v}(s))W(\gamma_{x,v}(s),\dot{\gamma}_{x,v}(s))\,ds.
$$
Applying the Hypothesis \ref{hypo} again, it yields $\p^k_\eta\hat{S}(0,\gamma_{x,v}(s))=0$.  
Therefore, $\p^k_\eta\hat{S}(0,\cdot)=0$ for all $k$. Since $S$ is compactly supported in $t$, its Fourier transform $\hat{S}(\eta,\cdot)$ is analytic with respect to $\eta$. Combining with $\p^k_\eta\hat{S}(0,\gamma_{x,v}(s))=0$ for all $k$, we conclude $\hat{S}(\eta,\cdot)=0$ for all $\eta\in\R$ and thus $S\equiv 0$.
\end{proof}

\begin{remark}
     The proof of Theorem \ref{injectivity of weighted light ray} relies essentially on the assumption that the unknown function $S$ has compact support, so the time Fourier transform of $S$ is analytic. Therefore it suffices to recover $\hat S$ and all its derivatives at single frequency $\eta=0$. In Appendix \ref{appendix 2}, we provide an alternative approach of showing the injectivity of light ray transforms of functions and tensor fields that are not necessarily compactly supported, by imposing hypothesises regarding the injectivity of certain weight geodesic X-ray transforms with respect to any frequencies $\eta\in\mathbb R$. 
\end{remark}

\subsection{Proof of main results} We are now ready to prove  Theorem~\ref{thm:main thm euclidean} and Theorem~\ref{thm:main thm geometric} as well as Corollary~\ref{cor:geometric condition}.
The two main theorems follow directly from the above facts. 
\begin{proof}[Proof of Theorem~\ref{thm:main thm euclidean} and Theorem~\ref{thm:main thm geometric}] 
    We apply Theorem \ref{injectivity of weighted light ray} to the light ray transform appearing in Proposition \ref{reduction to light ray} with $S=\widetilde{q}$ and $W(\gamma_{x,v}(s),\dot{\gamma}_{x,v}(s))= e^{-\int^s_0 (m-1)\sigma(\gamma_{x,v}(\ell),\dot{\gamma}_{x,v}(\ell))\,d\ell}$.
    Thus, the injectivity, $\widetilde{q}=q_1-q_2=0$, holds.
\end{proof}

For the geometric setting without scattering (namely, $\mu\equiv 0$ in \eqref{EQN: RTE equ}), the uniqueness result also holds in the manifold $(M,g)$ for any dimensions $n\geq 3$.
\begin{proof}[Proof of Theorem~\ref{thm:riemannian with scattering}]
    Without the presence of the scattering, that is, $\mu\equiv 0$, the remainder term $r_\lambda$ is trivial so that the GO solutions have relatively simple form, that is,
    $$
        u_\lambda(t,x,v)=\varphi_\lambda(t,x,v)\Theta_\sigma(x,v).
    $$
    As a result, by adjusting the arguments in previous subsections with $r_\lambda=0$, we will end up with the same weighted light ray transform in Proposition~\ref{reduction to light ray}. This enables us to use Theorem~\ref{injectivity of weighted light ray} again to derive $q_1=q_2$.
\end{proof}

Finally, we replace the injectivity assumption of the attenuated X-ray transform in Theorem \ref{thm:main thm euclidean} and \ref{thm:main thm geometric} by proper geometric conditions to establish Corollary \ref{cor:geometric condition}.

\begin{proof}[Proof of Corollary~\ref{cor:geometric condition}]
For the case (1), if $M$ is an open bounded and strictly convex domain in $\mathbb R^n$, then $\overline M$ is simple and admits a smooth strictly convex function (see e.g. \cite[Lemma 2.1]{PSUZ2019}). Since $\sigma\in C^\infty(S\overline M)$, if $I_{(m-1)\sigma}S=0$ for $S\in L^\infty(M)$, then $S\in C^\infty(\overline M)$, see e.g. \cite[Proposition 3]{FSU}. Now as $\overline M$ admits a smooth strictly convex function, it follows that the attenuated X-ray transform $I_{(m-1)\sigma}$ is injective on $L^\infty(M)$ by \cite{PSUZ2019}. Now we can apply Theorem \ref{thm:main thm euclidean} to conclude that $q_1=q_2$.

For the case (2), both assumptions (a) and (b) imply that $\overline M$ is a simple surface. Since $\sigma\in C^\infty(\overline M)$, similar to case (1), one can conclude that $\ker I_{(m-1)\sigma} \cap L^\infty(M)\subset C^\infty(\overline M)$. Therefore the attenuated X-ray transform $I_{(m-1)\sigma}$ is injective on $L^\infty(M)$ if $\sigma\in C^\infty(\overline M)$ by \cite{SU2011}. Now if $M$ is an open bounded and strictly convex domain in $\mathbb R^2$, we apply Theorem \ref{thm:main thm euclidean} to obtain that $q_1=q_2$. On the other hand, if $M$ is the interior of a simple surface and $\mu\in\mathcal M$, Theorem \ref{thm:main thm geometric} implies that $q_1=q_2$ as well.
\end{proof}

\subsection{Another uniqueness result based on monotonicity  
}
Before closing Section~\ref{sec:ICP}, we will study the problem under the monotonicity assumption $q_1\leq q_2$. With this new information, the uniqueness result can be proved directly from Lemma~\ref{lemma:GO id}.

\begin{theorem}[Monotonicity condition]  
Suppose that all conditions in Lemma~\ref{lemma:GO id} hold. Suppose that $q_1\leq q_2$.  
 If $\mathcal A_1(h_0,h_-)=\mathcal A_2(h_0,h_-)$ for all $(h_0,h_-)\in \mathcal{X}_\delta$, then 
$$
    q_1=q_2\quad \hbox{ in }(0,T)\times M.
$$
\end{theorem}
\begin{proof}
We take $\eta=m\lambda$ and  
\begin{align*}
    \varphi_\lambda (t,x,v) & =\phi(t-\tau_-(x,v),\rho_{x,v}(-\tau_-(x,v)))e^{i\lambda (t-\tau_-(x,v))},\\
    \varphi_\eta (t,x,v) & =\varphi_{m\lambda}(t,x,v)=\phi(t-\tau_-(x,v),\rho_{x,v}(-\tau_-(x,v)))e^{i m\lambda (t-\tau_-(x,v))}
\end{align*}
in \eqref{u_12}-\eqref{u_-}, where $\phi\in C^\infty_0(\R, C^\infty_0(\p_- SM))$ satisfying $\phi\geq 0$.

Since $q_2\geq q_1$, from \eqref{ID: GO}, we have
    \begin{align*}
        0&=\int_{SM_T}(q_2-q_1)(t,x) \bigg(\phi(t-\tau_-(x,v),\rho_{x,v}(-\tau_-(x,v)))\bigg)^{m+1} \Theta_\sigma^{m-1}   \,dtd\Sigma\\
        &\geq  e^{-(m-1)\sigma_0 diam(M)}\int_{SM_T}(q_2-q_1)(t,x) \bigg(\phi(t-\tau_-(x,v),\rho_{x,v}(-\tau_-(x,v)))\bigg)^{m+1} \,dtd\Sigma \geq 0.
    \end{align*}
Since $\p M$ is strictly convex, there exists $\phi\in C^\infty_0(\R, C^\infty_0(\p_- SM))$ such that 
$$
    \int_{S_xM}\bigg(\phi(t-\tau_-(x,v),\rho_{x,v}(-\tau_-(x,v)))\bigg)^{m+1}\,dv>0
$$
for all $(t,x)\in M_T$. This implies that for a.e. $(t,x)\in M_T$, 
$$
    (q_2-q_1)(t,x)=0.
$$
\end{proof}

\section{Inverse source problem}\label{sec:inverse source problem}
In this section, we consider the linear transport equation in the non-scattering medium. The objective is to reconstruct the source term in the equation from the measurement. Accordingly, we will first link the expression of the transport solution in \cite{LaiUhlmannZhou22} and the weighted light ray transform under suitable assumptions.

\subsection{A connection between different expressions of transport solutions}
Assume that $S\in L^\infty(\mathbb R\times M)$ is supported in $M_T$.
From Proposition~2.2 in \cite{LaiUhlmannZhou22}, the solution $f$ to the following equation
\begin{align}\label{EQN: f nonscattering}
    \p_t f + X f  + \sigma f = S(t,x) 
\end{align}
with $f|_{t=0}=f|_{\p_-SM_T}=0$ can be written as 
\begin{align}\label{ID: f integral formulation}
    f(t,x,v)  
	 = \int^t_0  e^{-\int^s_0 \sigma(\rho_{x,v}(-r))dr} S(t-s,\gamma_{x,v}(-s))H(\tau_-(x,v)-s)\,ds,\quad (t,x,v)\in SM_T,
\end{align}
where $H$ is the Heaviside function, that is, $H$ satisfies $H(s)=0$ if $s<0$ and $H(s)=1$ if $s>0$.

For $(x,v)\in SM$, when $t\geq \tau_-(x,v)$, it is clear that \eqref{ID: f integral formulation} can be expressed as
\begin{align}\label{ID: f integral formulation 1}
    f(t,x,v)  
	&= \int^{\tau_-(x,v)}_0  e^{-\int^s_0 \sigma(\rho_{x,v}(-r))dr} S(t-s,\gamma_{x,v}(-s)) \,ds.
\end{align}
On the other hand, when $t<\tau_-(x,v)$, since $S$ is supported in $M_T$ which implies $S(t-s,\gamma_{x,v}(-s))=0$ for $s\in [t,\tau_-(x,v))$, we can extend its integral region from $(0,t)$ to $(0,\tau_-(x,v))$ so that \eqref{ID: f integral formulation} is also of the form \eqref{ID: f integral formulation 1} for $t<\tau_-(x,v)$. 

Replacing $v$ by $-v$, we get from \eqref{ID: f integral formulation 1} that
\begin{align}\label{ID: f integral formulation 2}
    \tilde{f}(t,x,v) := f(t,x,-v)  
	&= \int^{\tau(x,v)}_0  e^{-\int^s_0 \sigma(\rho_{x,v}(r))dr} S(t-s,\gamma_{x,v}(s)) \,ds,
\end{align}
which is the weighted light ray transform satisfying
$$
    \p_t \tilde{f} - X \tilde{f} + \sigma \tilde{f} = S(t,x) \quad\hbox{ with } \tilde{f}|_{t=0}=\tilde{f}|_{\p_+SM}=0.
$$
  
\subsection{Inverse source problem in the absence of the scattering}
Recall that in both the well-posedness and linearization, we view the nonlinear coefficient $q$ as part of a source term for linear transport equations. By slightly modifying the arguments in Section~\ref{sec:linearization}-\ref{sec:main theorem}, we can prove the following uniqueness result for the inverse source problem for the time-dependent linear transport equation:
\begin{align}\label{EQN: RTE equ without scattering}
	\left\{\begin{array}{rcll}
		\p_tf + X f  + \sigma f &=& S(t,x) & \hbox{in } SM_T,  \\
		f &=&0 & \hbox{on } \{0\}\times SM,\\
		f &=&0& \hbox{on }\p_-SM_T.
	\end{array}\right.
\end{align}

We define the measurement operator by 
\begin{align*}
\mathcal{A}: S\in L^\infty(M_T)\rightarrow (f|_{t=T},f|_{\p_+SM_T})\in L^\infty(SM)\times L^\infty(\p_+SM_T),
\end{align*}
which is well-defined. The aim is to determine the source $S$ from measuring $\mathcal{A}(S)$. 

\begin{proposition}
Let $f_j$ be the solution to \eqref{EQN: RTE equ without scattering} with $S=S_j\in L^\infty(M_T)$ for $j=1,\,2$. Suppose that $\widetilde{S}:=S_1-S_2$. 
Suppose that Hypothesis~\ref{hypo} holds with $S=\widetilde{S}$ and $W=e^{-\int^s_0 \sigma(\gamma_{x,v}(\ell),\dot{\gamma}_{x,v}(\ell))\,d\ell}$.
If $\mathcal A(S_1) =\mathcal A(S_2)$, 
then 
$$
    S_1=S_2\quad\hbox{ in }\quad (0,T)\times M.
$$
\end{proposition}

\begin{proof}
We extend $\widetilde{S}$ by zero to $\R$ in the $t$ variable so that $\widetilde{S}\in L^\infty(\R\times M)$.
Due to the well-posedness, there exists a unique solution $f_j\in L^\infty(SM_T)$ to \eqref{EQN: RTE equ without scattering} with $S=S_j$ for $j=1,2$, respectively. Denote $F:=f_1-f_2$. Then $F$ is a solution to
\begin{align}\label{EQN:F}
    \p_t F + X F  + \sigma F=\widetilde{S}(t,x)
\end{align}
and satisfies $F|_{t=0}=F|_{\p_-SM_T}=0$.  
Without the scattering term, as mentioned above, the solution $F$ can be expressed as
\begin{align*}
    F(t,x,v)  
	&= \int^{\tau_-(x,v)}_0  e^{-\int^s_0 \sigma(\rho_{x,v}(-r))dr} \widetilde{S}(t-s,\gamma_{x,v}(-s)) \,ds
\end{align*}
and $F(t,x,-v)$ becomes a weighted light ray transform \eqref{ID: f integral formulation 2}, denoted by
$$
    \widetilde{F}(t,x,v):=F(t,x,-v) = \int^{\tau(x,v)}_0  e^{-\int^s_0 \sigma(\rho_{x,v}(r))dr} \widetilde{S}(t-s,\gamma_{x,v}(s)) \,ds.
$$
Moreover, from $\mathcal A(S_1)=\mathcal A(S_2)$, we also have $F|_{\p_+SM_T}=F|_{t=T}=0$, which leads to $\widetilde{F}|_{t=T}=0$ and $\widetilde{F}|_{\p_-SM_T}=F|_{\p_+SM_T}=0$. Thus
$$
    \widetilde{F}(t,x,v)=0\quad (t,x,v)\in (0,T)\times \p_-SM. 
$$
Indeed, $\widetilde{F}=0$ in $\R\times \p_-SM$. To see this, we observe that $\widetilde{F}(t,\cdot,\cdot)=0$ in $\p_-SM$ for $t <0$ since $\widetilde{S}$ is supported in $SM_T$. On the other hand, for case $t>T$, we let $t=T+t'$ with $t'>0$. Through the change of variable $s\mapsto s'=s-t'$, we have
\begin{align*}
    \widetilde{F} (T+t',x,v) 
    &= \int^{\tau(x,v)}_0  e^{-\int^s_0 \sigma(\rho_{x,v}(r))dr} \widetilde{S}(T+t'-s,\gamma_{x,v}(s)) \,ds \\
    &= \int^{\tau(\rho_{x,v}(t'))}_{-t'}  e^{-\int^{s'+t'}_0 \sigma(\rho_{x,v}(r))dr} \widetilde{S}(T-s',\gamma_{\rho_{x,v}(t')}(s')) \,ds'\\
    &= \int^{\tau(\rho_{x,v}(t'))}_{-t'}  e^{-\int^{s'}_{-t'} \sigma(\rho_{\rho_{x,v}(t')}(r'))dr'} \widetilde{S}(T-s',\gamma_{\rho_{x,v}(t')}(s')) \,ds'\\
    &= e^{-\int^{0}_{-t'} \sigma(\rho_{\rho_{x,v}(t')}(r'))dr'} \LC\int^{\tau(\rho_{x,v}(t'))}_{0}  e^{-\int^{s'}_{0} \sigma(\rho_{\rho_{x,v}(t')}(r'))dr'} \widetilde{S}(T-s',\gamma_{\rho_{x,v}(t')}(s')) \,ds'\RC\\
    &= e^{-\int^{0}_{-t'} \sigma(\rho_{\rho_{x,v}(t')}(r'))dr'} \widetilde{F}(T, \rho_{x,v}(t')),
\end{align*}
where we applied the change of variable $r\mapsto r'=r-t'$ and also used the fact that $\widetilde{S}$ is supported in $M_T$ in the third and fourth identities, respectively. When $t'>\tau(x,v)$, due to the support of $\widetilde{S}$, $\widetilde{F}(T, \rho_{x,v}(t'))=0$, which yields that $\widetilde{F} (T+t',x,v) =0$. When $0<t'< \tau(x,v)$, the flow $\rho_{x,v}(t')\in SM$. Since the final data $\widetilde{F}|_{t=T}=0$ in $SM$, we also get $\widetilde{F}(T, \rho_{x,v}(t'))=0$ and thus $\widetilde{F} (T+t',x,v) =0$. Now we have showed that $\widetilde{F}=0$ in $\R\times \p_-SM$ a.e.. 

Finally, under the hypothesis \ref{hypo} that the attenuated ray transform is injective, performing the same argument as in the proof of Theorem~\ref{injectivity of weighted light ray} yields $\widetilde{S}=0$.
\end{proof}

\appendix

\section{Light ray transform on infinite cylinder}\label{appendix 2}

In this section, we consider the light ray transform of functions or tensor fields that are not necessarily supported in $M_T$.  

We give an alternative approach of showing the injectivity, up to natural gauge, of the light ray transform of symmetric tensor fields. It has been studied in \cite{FIKO2021, FIO2021} for static Lorentzian manifolds using analyticity of the time Fourier transform of tensors compactly supported in time. Earlier results in Minkowski spacetime can be found in e.g. \cite{Ra2018, KSV2020}.

Recall that $M$ is the interior of a compact non-trapping Riemannian manifold $(\overline M, g)$, of dimension $n\geq 2$, with smooth strictly convex boundary $\p M$. Let $f\in L^\infty(N;\mathcal S^m)\cap L^1(\mathbb R; C(\overline M);\mathcal S^m)$
be a symmetric tensor field of rank $m$ on $N=\mathbb R\times M$, where $\mathcal S^m$ denotes the space of symmetric tensor fields of rank $m$ on $N$ for each $m=0,1,\cdots$. In local coordinates $z=(z^0,z^1,\cdots, z^n)=(t,x^1,\cdots,x^n)$, 
$$f(z)=f_{i_1\cdots i_m}(z) dz^{i_1}\cdots dz^{i_m},$$
where $dz^0=dt$ and $dz^j=dx^j$ for $j=1,\cdots, n$. Equivalently, we can write $f$ as
$$f=f_m+f_{m-1}dt+\cdots+f_1 (dt)^{m-1}+ f_0 (dt)^m,$$
where $f_k\in L^\infty(\mathbb R;L^\infty(M;S^k))\cap L^1(\mathbb R;C(\overline M;S^k))$ for $k=0,1,\cdots,m$. Here $S^k$ denotes the space of symmetric tensor fields of rank $k$ on $M$, and $f_k (dt)^{m-k}$ denotes the symmetrized tensor product of $f_k$ with $(dt)^{m-k}$. We denote the (static) Lorentzian metric on $N$ by $\bar g=-(dt)^2+g$. 

For simplicity, we denote the space $L^\infty(N;\mathcal S^m)\cap L^1(\mathbb R; C(\overline M);\mathcal S^m)$ by $\mathcal L(N;\mathcal S^m)$, and the space $L^\infty(\mathbb R;L^\infty(M;S^k))\cap L^1(\mathbb R;C(\overline M;S^k))$ by $\mathcal D(M;S^k)$. Note that $\mathcal L(N,\mathcal S^m)\subset L^2(N;\mathcal S^m)$.

\begin{lemma}\label{tensor decomposition}
    Given any $f\in \mathcal L(N; \mathcal S^m)$, there exist $p\in \mathcal D(M;S^{m})$, $q\in \mathcal D(M;S^{m-1})$, $r\in \mathcal D(M;S^{m-2})$ and $\lambda\in \mathcal L(N;\mathcal S^{m-2})$ such that
$$f=(p+r\,g)+q\, dt+\lambda \,\bar g.$$
\end{lemma}
\begin{proof}
We prove the above decomposition by induction. The cases $m=0, 1$ are obvious. When $m=2$, notice that
$$f=f_2+f_1 dt+f_0 (dt)^2=(f_2+f_0\,g)+f_1 dt+(-f_0)(-(dt)^2+g).$$
 Now assume that the decomposition holds for some $m\geq 2$, let $f$ be a tensor of rank $m+1$, then
 \begin{align*}
    f& =f_{m+1}+f_{m}dt+\cdots+f_0 (dt)^{m+1}\\
    &= f_{m+1}+(f_m+f_{m-1}dt+\cdots+f_0 (dt)^m) \, dt\\
    & =f_{m+1}+\big(p+r\,g+q\, dt+\lambda \,\bar g \big)\, dt\\
    &=f_{m+1}+(p+r\,g) dt+q (dt)^2+\lambda (dt) \bar g\\
    & =(f_{m+1}+q\,g)+(p+r\,g) dt+ (\lambda \,dt-q)\bar g.
\end{align*}  
\end{proof}

Consider the light ray transform of a symmetric tensor field $f$ of rank $m$ on $N=\mathbb R\times M$ defined by
$$L_m f(t,x,v)=\int_0^{\tau(x,v)} f_{i_1\cdots i_m}(\tilde\gamma_{t,x,v}(s))\dot{\tilde\gamma}_{t,x,v}^{i_1}(s)\cdots \dot{\tilde\gamma}_{t,x,v}^{i_m}(s)\, ds, \quad (t,x,v)\in \mathbb R\times S\overline{M},$$
where the light ray/null geodesic $\tilde\gamma_{t,x,v}(s)=(t+s,\gamma_{x,v}(s))$ with $\gamma_{x,v}$ the geodesic on $M$. Note that $\dot{\tilde\gamma}^0_{t,x,v}(s)\equiv 1$ and $\dot{\tilde\gamma}^j_{t,x,v}(s)=\dot\gamma^j_{x,v}(s)$ are the $j^{th}$-component of $\dot{\gamma}_{x,v}(s)$ for $j=1,\cdots,n$. It's easy to see that $L_m (\lambda\,\bar g)\equiv 0$ for any $\lambda\in \mathcal L(N;\mathcal S^{m-2})$.
We may simply denote the integrand of $L_mf$ by 
$$f(t+s,\gamma_{x,v}(s),\dot\gamma_{x,v}(s)).$$

Let $\eta\in \mathbb R$, we define the following attenuated geodesic X-ray transform on $M$ of $f\in L^\infty(M;S^k)$, $k\geq 0$, as
$$I_\eta f(x,v)=\int_0^{\tau(x,v)} e^{i\eta s} \, f_{i_1\cdots i_k}(\gamma_{x,v}(s))\dot\gamma_{x,v}^{i_1}(s)\cdots \dot\gamma_{x,v}^{i_k}(s) \,ds.$$
We denote $d^s$ the symmetric differentiation w.r.t. the Riemannian metric $g$.

\begin{hypothesis}\label{hypo 1}
    For each $m=0,1,\ldots$, we say that $I_\eta$ is s-injective for degree $m$ if $I_\eta f|_{\p_-SM}\equiv 0$, $f\in \bigoplus_{k=0}^m C(\overline M;S^k)$, implies that $f=d^s p+i\eta p$ for some $p\in \bigoplus_{k=0}^{m-1} C^1(\overline M;S^k)$ with $p|_{\p M}=0$. When $m=0$, this just means that $f=0$, i.e. $I_\eta$ is injective.
\end{hypothesis}

Notice that the attenuation $i\eta$ is complex. When $m=0$ and $1$, injectivity results for $I_\eta$ were proved on simple manifolds \cite{SU2011, GPSU2016, Zhou2017} and manifolds admitting strictly convex functions \cite{PSUZ2019}, see also recent survey \cite{ZhouSurvey2024} and the references therein. When $m\geq 2$, injectivity results are known on simple surfaces \cite{KMM2019} and negatively curved manifolds \cite{PS2023}.

\begin{theorem}\label{infinite tensor light ray}
    Assume that the attenuated X-ray transform $I_\eta$ is s-injective for degree $m$ for all $\eta\in \mathbb R$. Let $f\in L^\infty(N;\mathcal S^m)\cap L^1(\mathbb R; C(\overline M);\mathcal S^m)$. 
    If $L_m f=0$ for all $(t,x,v)\in \mathbb R\times \p_-SM$, then 
    \begin{itemize}
        \item $f=0$, if $m=0$;
        \item $f=\bar d^s p$ for some $p\in L^\infty(N)\cap H^1(\mathbb R;C^1(\overline M))$, $p|_{\p N}=0$, if $m=1$;
        \item $f=\bar d^s p+\lambda \bar g$ for some $p\in L^\infty(N;\mathcal S^{m-1})\cap H^1(\mathbb R;C^1(\overline M);\mathcal S^{m-1})$, $p|_{\p N}=0$, and $\lambda\in \mathcal L(N;\mathcal S^{m-2})$, if $m\geq 2$.
    \end{itemize} 
    Here $\bar d^s$ is the symmetric differentiation w.r.t. the Lorentzian metric $\bar g$, i.e., $\bar d^s=d^s+\p_t(\cdot)\,dt$.
\end{theorem}

\begin{proof}
By Lemma \ref{tensor decomposition} and the fact that $L_m(\lambda \bar g)\equiv 0$, we may assume that 
$$f=p+r\, g+q \,dt$$
for some $p\in \mathcal D(M;S^m)$, $q\in \mathcal D(M;S^{m-1})$ and $r\in \mathcal D(M;S^{m-2})$.
In the meantime, $f$ can be rewritten as
\begin{align}\label{tensor f}
f=p+q\,dt+r(dt)^2+r\,\bar g,
\end{align}
so we only consider $f$ of the form 
$$p+q\,dt+r(dt)^2.$$

      Following the calculation in the proof of Theorem \ref{injectivity of weighted light ray}, since $f\in \mathcal L(N;\mathcal S^m)$, by the Fubini's theorem
\begin{equation}\label{Fourier tensor light ray}
0=\widehat{L_m f}(\eta,x,v)=\int_0^{\tau(x,v)} e^{i\eta s}\hat f(\eta, \gamma_{x,v}(s),\dot\gamma_{x,v}(s))\,ds.
\end{equation}
Note that $\hat f\in L^2(N;\mathcal S^m)\cap C(\overline N;\mathcal S^m)$.
We denote $\hat f(\eta,\cdot)$ by $\hat f_\eta$, for each fixed $\eta\in\mathbb R$, we can view 
$$\hat f_\eta=\hat p_\eta+\hat q_\eta+\hat r_\eta\in C(\overline M;S^m)\oplus C(\overline M;S^{m-1})\oplus C(\overline M;S^{m-2}).$$
The right hand side of \eqref{Fourier tensor light ray} corresponds to the attenuated X-ray transform $I_\eta \hat f_\eta$.

Now by the Hypothesis \ref{hypo 1}, there exists $u_\eta\in \bigoplus_{k=0}^{m-1} C^1(\overline M;S^k)$ such that 
$$\hat p_\eta+\hat q_\eta+\hat r_\eta=d^s u_\eta+i\eta u_\eta$$
and $u_\eta|_{\p M}=0$.
Balancing both sides, we conclude that for $\eta\neq 0$, $u_\eta=w_\eta+v_\eta$ with $w_\eta\in C^1(\overline M;S^{m-1})$ and $v_\eta\in C^1(\overline M;S^{m-2})$, i.e.
\begin{equation}\label{Fourier equations}
\hat p_\eta=d^s w_\eta,\quad \hat q_\eta=d^s v_\eta+i\eta w_\eta,\quad \hat r_\eta=i\eta v_\eta.  
\end{equation}
Since $\hat p_\eta$, $\hat q_\eta$ and $\hat r_\eta$ are all $L^2$ in $\eta$,  
we can derive the regularity of $w_\eta$ and $v_\eta$ in $\eta$ from \eqref{Fourier equations} (note that we may freely assign values for $w_\eta$ and $v_\eta$ at $\eta=0$). In particular $w_\eta$, $\eta w_\eta$, $v_\eta$ and $\eta v_\eta$ are all $L^2$ in $\eta$. We denote the inverse Fourier transform of a function $g$ by $\check{g}$. It follows that
\begin{equation}\label{inverse Fourier equations}
    p=d^s \check w,\quad q=d^s \check v+\p_t \check w,\quad r=\p_t \check v,
\end{equation}
for some $\check w\in H^1(\mathbb R;C^1(\overline M;S^{m-1}))$ and $\check v\in H^1(\mathbb R;C^1(\overline M; S^{m-2}))$ with $\check w|_{\p N}=0$, $\check v|_{\p N}=0$. Moreover, the regularity of $p, q$ and $r$, together with \eqref{inverse Fourier equations}, implies that both $\check w$ and $\check v$ are $L^\infty$ in time $t$. 
Therefore by \eqref{tensor f}, with $\bar d^s=d^s+\p_t(\cdot)\,dt$,
\begin{align*}
    f & =p+q\,dt+r(dt)^2+r\,\bar g\\
      & =\bar d^s (\check w+\check v\, dt)+r\,\bar g,
\end{align*} 
where $\check w+\check v\, dt\in L^\infty(N;\mathcal S^{m-1})\cap H^1(\mathbb R;C^1(\overline M);\mathcal S^{m-1})$ vanishing on $\p N$.
\end{proof}

One can also consider the injectivity of the weighted light ray transform $L_W S$, defined in \eqref{def:weighted light ray}, for $S\in L^\infty(\mathbb R\times M)\times L^1(\mathbb R; C(\overline M))$.

\begin{theorem}\label{injectivity of infinite weighted light ray transform}
Let $W\in L^\infty(SM)$ be nonvanishing, suppose that the weighted geodesic X-ray transform of $f\in C(\overline M)$, with the weight $e^{i\eta s}W(\gamma_{x,v}(s),\dot\gamma_{x,v}(s))$, is injective for all $\eta\in\mathbb R$. Assume that $S\in L^\infty(\mathbb R\times M) \cap L^1(\mathbb R;C(\overline M))$. 
If $L_W S=0$ for a.e. $(t,x,v)\in\R\times\p_-SM$, then $S=0$. 
\end{theorem}

\begin{proof}
    Similar to Theorem \ref{infinite tensor light ray}, by the Fubini's theorem
    $$0=\widehat{L_W S}(\eta,x,v)=\int_0^{\tau(x,v)} e^{i\eta s}W(\gamma_{x,v}(s),\dot{\gamma}_{x,v}(s))\hat S(\eta,\gamma_{x,v}(s))\,ds.$$
    The right hand side is the weighted X-ray transform of $\hat S(\eta,\cdot)$. Now by the assumption, we get that $\hat S(\eta,x)=0$ for all $\eta\in \mathbb R$ and $x\in \overline M$, therefore $S=0$.
\end{proof}

\section*{Acknowledgements}
 Ru-Yu Lai is partially supported by the National Science Foundation (NSF) through grants DMS-2006731 and DMS-2306221. Hanming Zhou is partly supported by the NSF grants DMS-2109116 and DMS-2408369, and Simons Foundation Travel Support for Mathematicians.

\bibliographystyle{abbrv}

\bibliography{1Ntransport}

\end{document}